\documentclass{amsart}

\usepackage{amsmath}
\usepackage{amssymb}
\usepackage{amsthm}
\usepackage[foot]{amsaddr}
\usepackage{enumerate}
\usepackage{enumitem}
\usepackage[colorlinks=true]{hyperref}
\usepackage{xcolor}
\usepackage{dsfont}

\theoremstyle{plain}
\newtheorem{theorem}{Theorem}[section]
\newtheorem{proposition}[theorem]{Proposition}
\newtheorem{lemma}[theorem]{Lemma}
\newtheorem{corollary}[theorem]{Corollary}
\newtheorem*{theorem*}{Theorem}

\theoremstyle{definition}
\newtheorem{definition}{Definition}
\newtheorem*{definition*}{Definition}

\theoremstyle{remark}
\newtheorem{remark}{Remark}[section]
\newtheorem{claim}{Claim}[subsection]
\newtheorem{question}{Question}

\newcommand{\R}{\mathbb{R}}
\newcommand{\Q}{\mathbb{Q}}
\newcommand{\Z}{\mathbb{Z}}
\newcommand{\N}{\mathbb{N}}

\def\L{\Lambda}

\DeclareMathOperator{\supp}{supp}
\DeclareMathOperator{\GL}{GL}
\DeclareMathOperator{\SL}{SL}
\DeclareMathOperator{\Ad}{Ad}
\DeclareMathOperator{\Id}{Id}

\DeclareMathOperator{\Int}{int}

\DeclareMathOperator{\Comm}{Comm}

\title[Good models]{Closed approximate subgroups: \\
compactness, amenability and approximate lattices}
\author{Simon Machado}
\address{Institute for Advanced Study \\
              1 Einstein Dr, Princeton \\
              NJ 08540 \\
              USA} 
 \email{machado@ias.edu}
\begin{document}
\begin{abstract}
We investigate properties of closed approximate subgroups of locally compact groups, with a particular interest for approximate lattices i.e. those approximate subgroups that are discrete and have finite co-volume.  

We prove an approximate subgroup version of Cartan's closed-subgroup theorem and study some applications. We give a structure theorem for closed approximate subgroups of amenable groups in the spirit of the Breuillard--Green--Tao theorem. We then prove two results concerning approximate lattices: we extend to amenable groups a structure theorem for mathematical quasi-crystals due to Meyer; we prove results concerning intersections of radicals of Lie groups and discrete approximate subgroups generalising theorems due to Auslander, Bieberbach and Mostow. As an underlying theme, we exploit the notion of good models of approximate subgroups that stems from the work of Hrushovski, and Breuillard, Green and Tao. We show how one can draw information about a given approximate subgroup from a good model, when it exists.
\end{abstract}

 \maketitle
 \tableofcontents

  \section{Introduction}

As defined by Tao \cite{MR2501249}, a $K$-approximate subgroup of a group $G$ is a subset $A$ that is symmetric $A=A^{-1}$, contains the identity, and satisfies $AA \subset FA$ for some finite subset $F \subset G$ with $|F| \leq K$. In a seminal paper \cite{MR2833482} Hrushovski showed that  approximate subgroups are closely related to neighbourhoods of the identity in locally compact groups. Breuillard--Green--Tao \cite{MR3090256} subsequently used his ideas in combination with techniques originating from Gleason's and Yamabe's work on Hilbert's 5th problem to prove a theorem describing the structure of arbitrary finite approximate subgroups. Their work was extended in Kreitlon Carolino's thesis \cite{MR3438951} to handle approximate groups that are relatively compact neighbourhoods of the identity in an arbitrary locally compact group. 

The goal of this paper is to further investigate properties of \emph{infinite} approximate subgroups. We will prove several results in this direction. In particular we will show (Theorem \ref{Theorem: Closed-approximate-subgroup theorem general form} below) that closed approximate subgroups of locally compact groups are commensurable to homomorphic images of open approximate groups in a possibly different locally compact ambient group. This can be seen as an analogue of the classical theorem of Cartan asserting that closed subgroups of Lie groups are Lie subgroups. In turn this enables us to refine the main result of Kreitlon Carolino's thesis \cite{MR3438951} by removing the assumption that the relatively compact approximate groups studied there have non-empty interior (Theorem \ref{Theorem: Structure of compact approximate subgroups} below). We will also prove (Theorem \ref{Theorem: Structure amenable approximate subgroups}) a structure theorem in the spirit of the Breuillard--Green--Tao theorem (\cite{MR3090256}) for \emph{amenable} closed approximate subgroups of locally compact groups. This applies, in particular, to all closed approximate subgroups of amenable locally compact groups.

But most of the paper will be devoted to investigating a special class of approximate groups called \emph{approximate lattices}. These were first systematically studied in recent work of Bj\"{o}rklund and Hartnick \cite{bjorklund2016approximate}. The approximate lattices are those approximate subgroups $\Lambda$ of an ambient locally compact group $G$ that have finite co-volume (i.e. $G=\Lambda \mathcal{F}$ for some Borel subset $\mathcal{F}$ of $G$ of finite Haar measure) and are uniformly discrete (or equivalently the identity is isolated in $\Lambda^2$). Note that this definition is due to Hrushovski \cite{hrushovski2020beyond}, Bj\"{o}rklund and Hartnick first defined \emph{uniform} approximate lattices (when $\mathcal{F}$ is relatively compact) as well as a non-uniform version of this notion under the name of \emph{strong approximate lattice}, see \cite[Def. 4.9]{bjorklund2016approximate} and Subsection \ref{Subsection: Invariant hull, strong approximate lattices, cut-and-project schemes} below (the link between these three notions is investigated in \cite{machado2022discrete}). When $G$ is commutative these sets had been defined and studied already in the 1970's by Yves Meyer \cite{meyer1972algebraic};  they are a model of the so-called \emph{mathematical quasi-crystals} and have been much studied since, in particular in connection with mathematical physics \cite{MR3136260}.

%But most of the paper will be devoted to investigating a special class of approximate groups called \emph{approximate lattices}. These were introduced in recent work of Bj\"{o}rklund and Hartnick \cite{bjorklund2016approximate}. The uniform (i.e. co-compact) approximate lattices are defined as those symmetric subsets $\Lambda$ of an ambient locally compact group $G$ that are relatively dense (i.e. $G=\Lambda C$ for some compact subset $C$ of $G$) and very discrete in the sense that $\Lambda^3$ is uniformly discrete (or equivalently the identity is isolated in $\Lambda^6$). Bj\"{o}rklund and Hartnick made the simple yet remarkable observation that such $\Lambda$'s must be approximate subgroups of $G$. They also defined a non-uniform version of this notion under the name of \emph{strong approximate lattice}, see \cite[Def. 4.9]{bjorklund2016approximate} and Subsection \ref{Subsection: Invariant hull, strong approximate lattices, cut-and-project schemes} below. When $G$ is commutative these sets had been defined and studied already in the 1970's by Yves Meyer \cite{meyer1972algebraic};  they are the so-called \emph{mathematical quasi-crystals} and have been much studied since, in particular in connection with mathematical physics \cite{MR3136260}.

A simple way to construct approximate lattices is via a cut-and-project scheme, namely  the datum  $(G,H, \Gamma)$ of two
locally compact groups $G$ and $H$, and a lattice $\Gamma$ in $G\times H$, which projects injectively to $G$ and densely to $H$. Given a symmetric relatively compact neighbourhood of the
identity $W_0 \subset H$, one defines the \emph{model set} $P_0(G,H,\Gamma,W_0) := p_G ((G \times W_0) \cap \Gamma)$, where $p_G$ is the projection to the first factor. It is easy to see that a model set, or even any set commensurable to a model set, must be an approximate lattice. For more on cut-and-project sets we refer the reader to \cite{bjorklund2016aperiodic,bjorklund2017aperiodic}.

Central to our paper is an idea, exploited in \cite{MR2833482} and \cite{MR3090256} on the way to the structure theorems established there, which consists in defining a certain topology on the approximate group by means of what we will call here a \emph{good model}:

 \begin{definition}\label{Definition: Good models}
Let $\L$ be an approximate subgroup of a group $\Gamma$. A group homomorphism $f: \Gamma \rightarrow H$ with target a locally compact group $H$ is called a \emph{good model (of $(\L, \Gamma)$)} if:
\begin{enumerate}
\item $f(\L)$ is relatively compact;
 \item there is $U \subset H$ a neighbourhood of the identity such that $f^{-1}(U) \subset \L$.
\end{enumerate}
If $\Gamma$ is generated by $\L$, then we say that $f$ is a good model of $\L$. Any approximate subgroup commensurable to $\L$ will be called a \emph{Meyer subset}.
\end{definition}

This definition is reminiscent of the construction of the so-called \emph{Schlichting completion} of a pair $(\Gamma,\Lambda)$, where $\Lambda \leq \Gamma$ are discrete groups such that $\Gamma$ commensurates $\Lambda$. Indeed if $U$ is also a compact subgroup, then $\Lambda:=f^{-1}(U)$ is commensurated by $\Gamma$. See for instance the work of Tzanev \cite{MR2015025} on Hecke pairs or the works of Shalom and Willis \cite{shalom2013commensurated} and of Caprace and Monod \cite[\S 5D]{MR2574741} on commensurators of discrete groups, where this construction plays a key role. 

We will see that good models exist in many situations of interest. Indeed our first observation is that for approximate lattices in an ambient group $G$ the existence of a good model is equivalent to being commensurable to a model set via a certain cut-and-project scheme $(G,H,\Gamma)$.

\begin{proposition}\label{Proposition: An approximate lattice has a good model if and only if it is a model set}
 Let $\L$ be an approximate lattice in a  locally compact second countable group $G$. Then $\L$ is a Meyer subset if and only if it is contained in and commensurable to a model set. 
\end{proposition}

Meyer \cite{meyer1972algebraic} (see also \cite{schreiber1973approximations,moody1997meyer,fish2019extensions}) showed that every approximate lattice in a locally compact commutative group $G$  comes from a \emph{cut-and-project} construction: in other words it is commensurable to a model set, or equivalently thanks to Proposition \ref{Proposition: An approximate lattice has a good model if and only if it is a model set}, it is a Meyer subset. The question of extending Meyer's theorem to other groups has been raised by Bj\"{o}rklund and Hartnick in \cite[Problem 1]{bjorklund2016approximate}. This has been achieved for nilpotent and solvable Lie groups in our previous works \cite{machado2020approximate} and \cite{machado2019infinite} following a method close in spirit to Meyer's. A consequence of the tools developed in the present paper will be a new proof of this fact and indeed a generalization to all locally compact amenable ambient groups (Theorem \ref{Theorem: Approximate lattices in amenable locally compact groups are contained in model sets} below). In a companion paper (\cite{machado2020apphigherrank}) we show, using some key results of the present paper (notably Proposition \ref{Proposition: An approximate lattice has a good model if and only if it is a model set} and Theorem \ref{Theorem: Closed-approximate-subgroup theorem general form}) in combination with Zimmer's cocycle superrigidity theorem, that Meyer's theorem also holds for strong approximate lattices in semi-simple Lie groups without rank one factors (we also mention Hrushovski's \cite{hrushovski2020beyond} which generalises Meyer's theorem to semi-simple groups via a different approach). Another consequence will be a proof of the analogue for approximate lattices of the classical facts about hereditary properties of lattices with respect to intersections with closed normal subgroups (Proposition \ref{Proposition: Intersection approximate lattice with a closed subgroup}) and in particular an Auslander-type theorem regarding the intersection with the amenable radical (Theorem \ref{Theorem: Auslander's theorem for approximate lattices}).

The existence of a good model can be established via the following fact:

\begin{proposition}\label{Theorem: Characterisation of good models short version}
 Let $\L$ be an approximate subgroup of a group $G$. There exists a good model $f$ of $\L$ if and only if there is a sequence of approximate subgroups $(\L_n)_{n \geq 0}$ such that $\L_0=\L$ and for all integers $n \geq 0$ we have $\L_{n+1}^2 \subset \L_n$ and $\L_n$ commensurable to $\L$. 
\end{proposition}

This observation is essentially folklore and already contained in \cite{MR2833482} (see also \cite[\S 6]{MR3090256}), but we will give a simple proof requiring only knowledge of elementary point-set topology, see the discussion after the statement of Theorem \ref{Theorem: Characterisation of good models}.

In contrast with Proposition \ref{Theorem: Characterisation of good models short version} we observe that non-trivial quasi-morphisms \emph{à la} Brooks \cite{brooks1981some} yield constructions of approximate groups that are not Meyer subsets:

\begin{theorem}\label{Theorem: Example of approximate subgroup without a good model}
 Let $F_2$ be the free group over two generators $a$ and $b$. For any two reduced words $w,x \in F_2$ define $o(x,w)$ as the number of occurrences of $w$ in $x$. Then for any $w \in F_2\setminus\left\{a,b,a^{-1},b^{-1}, e\right\}$ of length $l$ the set 
 $$ \left\{ g \in F_2  : |o(g,w)-o(g,w^{-1})| \leq 3l\right\} $$
 is an approximate subgroup but not a Meyer subset. 
\end{theorem}

We are now ready to state the main results of this paper. We consider first two interesting classes of approximate groups: compact approximate subgroups and amenable approximate subgroups (Definition \ref{Definition: Amenable approximate subgroups}). We will see that these types of approximate subgroups always have good models, and thus are particularly regular types of approximate subgroups. In the case of compact approximate subgroups this leads to a closed-subgroup theorem for approximate subgroups.  

\begin{theorem}[Closed-approximate-subgroup theorem]\label{Theorem: Closed-approximate-subgroup theorem Lie group form}
Let $\L$ be a closed approximate subgroup of a locally compact group $G$. There are a locally compact group $H$, an injective group homomorphism $\phi:H \rightarrow G$ and an open approximate subgroup $\Xi$ of $H$ such that $\phi_{|\phi^{-1}(\L)}$ is proper and $\L \subset \phi(\Xi) \subset \L^3$. Furthermore, if $G$ is a Lie group, then $H$ is a Lie group.
\end{theorem}

Here, a good model of some compact approximate subgroup contained in $\L^2$ appears implicitly as the inverse of the map $\phi$. Theorem \ref{Theorem: Closed-approximate-subgroup theorem Lie group form} and a theorem of Schreiber \cite{schreiber1973approximations} (which was recently given a new proof by Fish \cite{fish2019extensions}) show that, modulo a compact error term, the structure of closed approximate subgroups of Euclidean spaces is akin to the structure of closed subgroups. Recall that a subset $\Lambda$ of an ambient group $G$ is called \emph{uniformly  discrete} if $e$ is isolated in $\Lambda^{-1}\Lambda$.  

\begin{proposition}
 Let $\L$ be a closed approximate subgroup in $\R^d$. Then we can find a vector subspace $V_o \subset \R^k$, as well as a uniformly discrete approximate subgroup $\L_d$ and a compact approximate subgroup $K_e$ both in a supplementary subspace $V_d $ of $V_o$ such that $\L$ is commensurable to $V_o + \L_d + K_e$. 
\end{proposition}

Theorem \ref{Theorem: Closed-approximate-subgroup theorem Lie group form} (in fact the more general Theorem \ref{Theorem: Closed-approximate-subgroup theorem general form}) also enables us to remove an openness assumption from a result of Kreitlon Carolino's thesis (\cite[Theorem 1.25]{MR3438951}) which generalised the main result of \cite{MR3090256}. This yields a precise structure theorem for all compact approximate subgroups. 

 \begin{theorem}[Structure of compact approximate subgroups]\label{Theorem: Structure of compact approximate subgroups}
  Let $\L$ be a compact approximate subgroup of a Hausdorff topological group $G$ and let $\langle\L\rangle$ be the subgroup it generates. Then $\langle\L\rangle$ admits a structure of locally compact group such that the inclusion $\langle\L\rangle \subset G$ is continuous and $\L \subset \langle\L\rangle$ is a compact subset with non-empty interior. Moreover, for every $\varepsilon > 0$ there is an approximate subgroup $\L' \subset \L^{16}$ that generates a subgroup open in the topology of $\langle\L\rangle$ and a compact subgroup $H \subset \L'$ normalised by $\L'$ such that:
  \begin{enumerate}[label=(\roman*)]
   \item $\L$ can be covered by $O_{K,\varepsilon}(1)$ translates of $\L'$;
   \item $\langle \L' \rangle/ H$ is a Lie group of dimension $O_K(1)$.
  \end{enumerate}
  If $\mathfrak{l}'$ denotes the Lie algebra of $\langle \L' \rangle/ H$ and $\L''$ the image of $\L'$ in $\langle \L' \rangle/ H$, then there exists a norm $\vert \cdot \vert$ on $\mathfrak{l}'$ such that:
  \begin{itemize}
   \item[(iii)] for $X,Y \in \mathfrak{l}'$ we have $\vert [X,Y] \vert\leq O_K(\vert X \vert \vert Y\vert)$  ;
   \item[(iv)] for $g \in \L''$ the operator norm (induced by $\vert\cdot \vert$) of $\Ad(g)-\Id$ is $O_K(\varepsilon)$;
   \item[(v)] there is a convex set $B \subset \mathfrak{l}''$ such that $\L''/B$ is a finite $O_{K,\varepsilon}(1)$-approximate local group. 
  \end{itemize}
\end{theorem}

Likewise, we are able to prove that amenability assumptions force approximate subgroups to have good models. It is well known that in many situations existence of some invariant finitely additive measures implies existence of a good model (see e.g. \cite{MR2833482, hrushovski2019amenability, MR2911137, MR2738997}). We will say that a closed approximate subgroup of a locally compact group $G$ is \emph{amenable} if there exists an invariant finitely additive probability measure on Borel subsets of $\Lambda$ (Section \ref{Section: Amenable}). This condition is in particular satisfied on any closed approximate subgroup close to an amenable normal subgroup (Proposition \ref{Theorem: Uniformly discrete approximate subgroups around an amenable approximate subgroup are amenable}). This enables us to prove a generalisation of Meyer's seminal theorem \cite{meyer1972algebraic}: 

\begin{theorem}[Meyer theorem for amenable groups]\label{Theorem: Approximate lattices in amenable locally compact groups are contained in model sets}
 If $\L$ is an approximate lattice in an amenable locally compact $\sigma$-compact group $G$, then $\L$ is contained in and commensurable to a model set.  
\end{theorem}

Recall that we had already established this result in \cite{machado2020approximate,machado2019infinite} in the special case when $G$ is a soluble Lie group. Our method here is very different however and is inspired from the work of Hrushovski \cite{MR2833482} and Breuillard--Green--Tao \cite{MR3090256}.

It is interesting to show that an approximate subgroup is amenable beyond the realm of approximate lattices. By using the strong Tits' alternative due to Breuillard \cite{breuillard2008strong} (and a consequence of it due to Breuillard, Green and Tao \cite{breuillard2011note}), we can prove a result reminiscent of the structure theorem of finite approximate subgroups \cite{MR3090256} (see also Proposition \ref{Proposition: Structure amenable approximate subgroups good model form} for a more complete statement in the language of good models). 

\begin{theorem}[Structure of amenable approximate subgroups]\label{Theorem: Structure amenable approximate subgroups}
Let $\Lambda$ be an amenable closed approximate subgroup of second countable locally compact group $G$. Then there is a closed approximate subgroup $\Lambda_{sol} \subset \overline{\Lambda^4}$ and a closed subgroup $N \subset \Lambda_{sol}$ such that:
\begin{enumerate}
\item $N$ is normal in $\langle \Lambda_{sol} \rangle$ and $\langle \Lambda_{sol} \rangle /N$ is a soluble group;
\item if, moreover, $\Lambda$ is $K$-approximate and $\langle \Lambda_{sol} \rangle$ is equipped with the topology given by Theorem \ref{Theorem: Closed-approximate-subgroup theorem Lie group form}, then $\langle \Lambda_{sol} \rangle /N$ is a Lie group of dimension bounded by $18 \log_2(K)$;
\item There is a compact neighbourhood $V$ of the identity in $\Lambda^n$ (in the induced topology) for some $n \geq 0$ such that $\Lambda$ is contained in $V\Lambda_{sol}$ and $V\Lambda_{sol} \cup  \Lambda_{sol}V$ is an approximate subgroup commensurable to $\Lambda$.
\end{enumerate}
\end{theorem} 

Subgroups, compact approximate subgroups (see Section \ref{Section: A closed-approximate-subgroup theorem}) and approximate subgroups of soluble lie groups (see \cite{machado2019infinite}) are natural and well-studied examples of amenable approximate subgroups. Theorem \ref{Theorem: Structure amenable approximate subgroups} asserts conversely that any amenable closed approximate subgroup of a locally compact group is built as a combination of these. We briefly mention two facts to illustrate the strength of Theorem \ref{Theorem: Structure amenable approximate subgroups}: when $G$ is supposed totally disconnected or $\Lambda$ is supposed uniformly discrete (e.g. when $\Lambda$ is an approximate lattice), we can choose $\Lambda_{sol}$ commensurable to $\Lambda$ (Corollaries \ref{Corollary: structure of discrete amenable approximate subgroups} and \ref{Corollary: structure of amenable approximate subgroups of totally disconnected groups}). Then $\Lambda$ is an extension of an amenable group by a soluble approximate subgroup.

Finally, we will use the ideas behind Theorem \ref{Theorem: Structure amenable approximate subgroups} to study a generalisation of theorems due to Auslander \cite[Theorem 1]{MR152607} and Mostow \cite[Lemma 3.9]{MR0289713} about intersections of lattices and radicals in Lie groups. 

\begin{theorem}\label{Theorem: Auslander's theorem for approximate lattices}
 Let $\L$ be an approximate subgroup in a locally compact group $G$. Suppose that there exists an amenable closed normal  subgroup $A$ such that $G/A$ is a finite product of simple algebraic groups over local fields. If the projections of $\langle\L\rangle$ to all simple factors of $G/A$ are Zariski-dense, then the projection of $\Lambda$ to $G/A$ is uniformly discrete.
\end{theorem}

When specialized to approximate lattices, we answer a question of Hrushovski \cite[Question 7.11]{hrushovski2020beyond}:

\begin{corollary}[Auslander--Mostow-type theorem for approximate lattices]\label{Corollary: Auslander's theorem for approximate lattices}
Let $\Lambda$ be an approximate lattice in a locally compact group $G$. Let $A$ be an amenable closed normal subgroup and suppose that $G/A$ is a finite product of simple algebraic groups over local fields. Suppose also that the projections of $\langle\L\rangle$ to all compact simple factors of $G/A$ are Zariski-dense. Then: 
\begin{enumerate}
\item $\Lambda^2 \cap A$ is an approximate lattice in $A$;
\item the projection of $\Lambda$ to $G/A$ is an approximate lattice in $G/A$.
\end{enumerate}
\end{corollary}
Corollary \ref{Corollary: Auslander's theorem for approximate lattices} enables us to decompose many approximate lattices into a semi-simple part and an amenable part. In light of Theorem \ref{Theorem: Approximate lattices in amenable locally compact groups are contained in model sets} and \cite{hrushovski2020beyond,machado2020apphigherrank}, both parts are known to be Meyer subsets. This invites us to wonder: 

\begin{question}\label{Question: Good models stable by extensions}
 With the notations of Corollary \ref{Corollary: Auslander's theorem for approximate lattices}. Let $p: G \rightarrow G/A$ denote the natural projection. We know that both $p(\L)$ and $\L^2\cap A$ are contained in model sets. Is $\L$ contained in a model set?  
\end{question}

In \cite{hrushovski2020beyond} is provided an example of an approximate lattice in a central extension of $\SL_2(\R) \times \SL_2(\Q_p)$ by $\Q_p$ with no good model, showing that the answer to Question \ref{Question: Good models stable by extensions} can be negative in some instances. A general answer to Question \ref{Question: Good models stable by extensions} should therefore take the (Lie, algebraic, connected, etc) structure of the ambient group $G$ into account. See  \cite[Question 7.12]{hrushovski2020beyond} for a related question and discussions around this topic. 

\subsection*{Structure of the paper}
In section \ref{Section: Preliminaries} we recall a few useful facts and definitions about approximate subgroups and approximate lattices. We then study general properties of good models in Section \ref{Section: Good models : definition, first properties and examples}. Using these tools we establish in Section \ref{Section: A closed-approximate-subgroup theorem} and Section \ref{Section: Amenable} the structure theorems for, respectively, compact approximate subgroups and amenable approximate subgroups. In Section \ref{Section: Generalisation of theorems of Mostow and Auslander} we study intersections of approximate lattices with closed subgroups, eventually proving Theorem \ref{Theorem: Auslander's theorem for approximate lattices}.

 \section{Preliminaries}\label{Section: Preliminaries}
 
 \subsection{Preliminaries on approximate subgroups and commensurability}
 We start with some notations: for a subset $X$ of a group $G$ and a non-negative integer $n$ define $X^{-1}= \{x^{-1} | x \in X \}$, $X^n := \{x_1\cdots x_n | x_1,\ldots,x_n \in X\}$ and $\langle X \rangle$ the group generated by $X$. Recall that an \emph{approximate subgroup} is a subset $\L$ of a group that is symmetric (i.e. $\L=\L^{-1}$), contains the identity and such that there exists a finite subset $F \subset G$ with $\L^2:=\{\lambda_1\lambda_2 \in G | \lambda_1,\lambda_2 \in \L\} \subset F\L$. We will say that two subsets $X,Y \subset G$ are \emph{(left-)commensurable} if there exists a finite subset $F \subset G$ such that $X \subset FY$ and $Y \subset FX$. Note that commensurability is an equivalence relation between subsets of a group. An approximate subgroup is thus a symmetric subset $\L$ containing $e$ such that $\Lambda^2$ is commensurable to $\L$. By an easy induction, we see moreover that $\Lambda^n$ is commensurable to $\Lambda$ for all $n \geq 1$.  We will denote by $\Comm_G(X)$ the subgroup of elements $g$ of $G$ such that $gXg^{-1}$ is commensurable with $X$.

 We collect here well-known facts about approximate subgroups and commensurability in a form and with hypotheses suitable to our discussion (See \cite{MR2501249,MR3090256,MR2833482,tointon_2019} for this and more background material). 
 
 \begin{lemma}[Ruzsa's covering lemma]\label{Lemma: Rusza's covering lemma}
  Let $X,Y$ be subsets of a group $G$ and $F \subset X$ be maximal such that $(fY)_{f \in F}$ is a family of disjoint sets. Then $X \subset FYY^{-1}$.
 \end{lemma}
 
 \begin{proof}
  If $x \in X$, then there is $f \in F$ such that $xY \cap fY \neq \emptyset$. So $x \in FYY^{-1}$.
 \end{proof}

 \begin{lemma}\label{Lemma: Intersection of commensurable subsets}
  Let $X_0,X_1,\ldots,X_r$ be subsets of a group $G$ and $F_1, \ldots, F_r \subset G$ be finite subsets such that $X_0 \subset F_iX_i$ for all integers $1 \leq i \leq r$. There is $F \subset G$ with $\vert F\vert \leq \vert F_1\vert \cdots \vert F_r\vert$ such that 
  $$X_0 \subset F\cdot \bigcap_{1 \leq i \leq n}  X_i^{-1}X_i.$$ 
 \end{lemma}

 \begin{proof}
  Take $f:=(f_i) \in F_1\times \cdots \times F_r$ and whenever $\bigcap_{1\leq i \leq r} f_iX_i \neq \emptyset$ choose an element $x_f \in \bigcap_{1\leq i \leq r} f_iX_i $. If $x$ is any element of $X_0$ then there must be some $f \in F_1\times \cdots \times F_r$ such that  $x \in \bigcap_{1\leq i \leq r} f_iX_i$. We thus have $x_f^{-1}x \in\bigcap_{1 \leq i \leq n}  X_i^{-1}X_i $. Defining $F:=\{x_f\vert f \in F_1 \times \cdots \times F_n, X_0 \cap \bigcap_{1\leq i \leq r} f_iX_i \neq \emptyset\}$, we find  
  $$X_0 \subset F\cdot \bigcap_{1 \leq i \leq n}  X_i^{-1}X_i.$$
 \end{proof}

 \begin{lemma}\label{Lemma: Intersection of approximate subgroups}
 Let $K_1,\ldots,K_r$ be positive integers and take a $K_i$-approximate subgroup $\L_i$ of $G$ for all $1 \leq i \leq r$. We have:
 \begin{enumerate}
  \item $\bigcap_{1\leq i\leq r} \L_i^2$ is a $K_1^3\cdots K_r^3$- approximate subgroup;
  \item if $(\Xi_i)_{1\leq i \leq r}$ is a family of approximate subgroups with $\Xi_i$ commensurable to $\L_i$ for all $1 \leq i \leq r$, then $\bigcap_{1\leq i \leq r} \L_i^2$ and $\bigcap_{1\leq i \leq r} \Xi_i^2$ are commensurable.
 \end{enumerate}
\end{lemma}

 \begin{proof}
 We know that $\L_i^4$ is covered by $K_i^3$ left-translates of $\L_i$ for all $1\leq i \leq r$. So (1) is a consequence of Lemma \ref{Lemma: Intersection of commensurable subsets} applied to $X_0=(\bigcap_{1\leq i \leq r}\L_i)^4$ and $X_1=\Lambda_1,\ldots,X_r=\Lambda_r$. To prove (2), it suffices to show that $\bigcap_{1\leq i \leq r} \L_i^2$ is covered by finitely many translates of $\bigcap_{1\leq i \leq r} \Xi_i^2$ by symmetry. Statement (2) is then a consequence of Lemma \ref{Lemma: Intersection of commensurable subsets} applied to $X_0=\bigcap_{1\leq i \leq r} \L_i^2$ and $\Xi_1,\ldots, \Xi_r$.
 \end{proof}

 \begin{lemma}\label{Lemma: Pull-back of commensurable approximate subgroups}
  Let $\L_1$ and $\L_2$ be two commensurable approximate subgroups of a group $G$. Let $\phi: H\rightarrow G$ be a group homomorphism. Then $\phi^{-1}(\L_1^2)$ and $\phi^{-1}(\L_2^2)$ are commensurable approximate subgroups of $H$. 
 \end{lemma}
 
 \begin{proof}
  By Lemma \ref{Lemma: Intersection of approximate subgroups} the subsets $\phi(H) \cap \L_1^2$ and $\phi(H)\cap \L_2^2$ are commensurable approximate subgroups. Take $\{i,j\} \subset \{1,2\}$, we can find a finite subset $F_{ij} \subset H$ such that 
 $ (\phi(H)\cap \L_i^2)^2  \subset \phi(F_{ij}) \left(\phi(H)\cap \L_j^2\right).$
 In other words, 
 $$ \phi^{-1}(\L_i^2)^2  \subset F_{ij} \phi^{-1}(\L_j^2).$$
 So $\phi^{-1}(\L_1^2)$ and $\phi^{-1}(\L_2^2)$ are commensurable approximate subgroups.
 \end{proof}

\subsection{Strong approximate lattices and cut-and-project schemes}\label{Subsection: Invariant hull, strong approximate lattices, cut-and-project schemes}

 Let $G$ be a locally compact second countable group and let $\mathcal{C}(G)$ be the set of closed subsets of $G$. The \emph{Chabauty-Fell} topology on $\mathcal{C}(G)$ is defined by the subbase of open subsets:
$$ U^V = \{ F \in \mathcal{C}(G) | F \cap V \neq \emptyset \} \text{ and } U_K = \{ F \in \mathcal{C}(G) \vert F \cap K = \emptyset \} $$
for all $V \subset G$ open and $ K \subset G$ compact. One can check that the map 
\begin{align*}
 G \times \mathcal{C}(G) & \rightarrow \mathcal{C}(G) \\
 (g, F) & \mapsto gF
\end{align*}
defines a continuous action of the group $G$ on $\mathcal{C}(G)$ and that $\mathcal{C}(G)$ is a compact metrizable set (see \cite{MR139135}). Convergence in the Chabauty-Fell topology can also be characterised in the following way: a sequence $(F_i)_{i\geq 0}$ converges to $F \in \mathcal{C}(G)$ if and only if (1) for every $x \in F$ there are $x_i \in F_i$ for all $i \in \N$ such that $x_i \rightarrow x$ as $i \rightarrow \infty$; (2) If $x_i \in F_i$ for all $i \in \N$ then every accumulation point of $(x_i)_{i \geq 0}$ lies in $F$ (see \cite[Section 2.2]{bjorklund2019borel}). 
 Given a closed subset $F$ of $G$ we define the \emph{invariant hull} $\Omega_F$ of $F$ as the closure of the $G$-orbit of $F$ (i.e. $\overline{G\cdot F}$) equipped with the induced continuous $G$-action. Note that if $H$ is a closed subgroup, then $\Omega_{H}$ is isomorphic as a compact $G$-space to the one-point compactification of $G/H$ (\cite[Lem. 2.3]{bjorklund2019borel}).

\begin{definition}[\cite{bjorklund2016approximate}]\label{Definition: Strong approximate lattice}
 Let $\L$ be an approximate subgroup of a locally compact second countable group $G$. We say that $\L$ is a \emph{strong approximate lattice} if: 
 \begin{enumerate}
  \item $\L$ is uniformly discrete i.e. $\L\L^{-1} \cap V = \{e\}$ for some neighbourhood of the identity $V$; 
  \item there is a $G$-invariant Borel probability measure $\nu$ on $\Omega_{\L}$ with  $\nu(\{\emptyset\})=0$ (we say that $\nu$ is \emph{proper}). 
 \end{enumerate}
\end{definition}

 In particular, a subgroup is a lattice if and only if it is a strong approximate lattice. Recall that a \emph{uniform approximate lattice} is a uniformly discrete approximate subgroup of $G$ such that there exists a compact subset $K \subset G$ with $\L K = G$. We do not know in general if uniform approximate lattices are strong. However, when the ambient group is amenable we have:

\begin{lemma}[Rem. 4.14.(1), \cite{bjorklund2016approximate}]\label{Lemma: Approximate lattices in amenable groups are strong}
 If $\L$ is a uniform approximate lattice in an amenable locally compact second countable group $G$, then $\L$ is a strong approximate lattice. 
\end{lemma}

We proved in \cite[\S 2.2.8]{machado2022discrete} that all strong or uniform approximate lattices are approximate lattices but it is not known whether the converse holds or not. Nonetheless, when the ambient group is amenable the notions of strong approximate lattices and approximate lattices are equivalent (\cite[Lem. 2.2.32]{machado2022discrete}).
 
 We recall now the definition of a cut-and-project scheme:

\begin{definition}[Definitions 2.11 and 2.12, \cite{bjorklund2016approximate}]\label{Definition: Cut-and-project scheme}\label{Definition: Model sets and Meyer sets}
 A \emph{cut-and-project scheme} is a triple $(G,H,\Gamma)$ consisting of two locally compact groups $G$ and $H$ and a lattice $\Gamma$ in $G\times H$ which projects injectively to $G$ and densely to $H$. For any symmetric relatively compact neighbourhood of the identity $W_0 \subset H$ we define the \emph{model set} 
 $$ P_0(G,H,\Gamma,W_0):= p_G\left(\left(G \times W_0\right) \cap \Gamma\right) \subset G $$
 where $p_G:G \times H \rightarrow G$ denotes the natural projection.
\end{definition}

It was shown in \cite{bjorklund2016approximate,bjorklund2016aperiodic} that cut-and-project schemes enable us to build strong approximate lattices:

\begin{proposition}[Theorem 3.4, \cite{bjorklund2016aperiodic} and Proposition 2.13, \cite{bjorklund2016approximate}]
 Let $(G,H,\Gamma)$ be a cut-and-project scheme and $W_0$ be a symmetric relatively compact neighbourhood of the identity. If $G$ is second countable and $\partial W_0$ is Haar-null, then $P_0(G,H,\Gamma,W_0)$ is a strong approximate lattice. If $\Gamma$ is a uniform lattice, then $P_0(G,H,\Gamma,W_0)$ is a uniform approximate lattice. 
\end{proposition}

\section{Good models : definition, first properties and examples}\label{Section: Good models : definition, first properties and examples}

In this section we investigate elementary properties of good models. We will prove in particular Proposition \ref{Proposition: An approximate lattice has a good model if and only if it is a model set}, Theorem \ref{Theorem: Characterisation of good models short version} and Theorem \ref{Theorem: Example of approximate subgroup without a good model}. 
\subsection{About the definition of good models}
Let us recall the definition of good models:

\begin{definition*}

Let $\L$ be an approximate subgroup of a group $\Gamma$. A group homomorphism $f: \Gamma \rightarrow H$ with target a locally compact group $H$ is called a \emph{good model (of $(\L, \Gamma)$)} if:
\begin{enumerate}
\item $f(\L)$ is relatively compact;
 \item there is $U \subset H$ a neighbourhood of the identity such that $f^{-1}(U) \subset \L$.
\end{enumerate}
\end{definition*}

\begin{remark}
 Restricting the range of the good model $f$ we can always assume that $f$ has dense image. 
\end{remark}

Definition \ref{Definition: Good models} involves both the choice of a map $f$ and an open subset $U$. However, up to commensurability, the choice of $U$ does not matter as the following shows: 

\begin{lemma}\label{Lemma: Inverse images of compact neighbourhoods are pairwise commensurable}
 Let $H$ be a locally compact group, $\Gamma$ be a discrete group,  $V_1$ and $V_2$  be symmetric relatively compact neighbourhoods of the identity in $H$ and $f:\Gamma \rightarrow H$ be a group homomorphism. The subsets  $f^{-1}(V_1)$ and $f^{-1}(V_2)$ are commensurable approximate subgroups. 
\end{lemma}

\begin{proof}
 Take $i,j \in \{1,2\}$. The identity belongs to the interior $\Int(V_j)$ of $V_j$ so
 $$ V_i^2 \subset\bigcup\limits_{h \in V_i^2} h\Int(V_j).$$
 But $\Int(V_j)$ is open and $V_i^2$ is relatively compact and, thus, there is a finite subset $F_{ij} \subset V_i^2$ such that $V_i^2 \subset F_{ij}U$. Since $V_1$ and $V_2$ are moreover symmetric subsets, we have that $V_1$ and $V_2$ are commensurable approximate subgroups. Choose now a symmetric open neighbourhood of the identity $W$ such that $W^2$ is contained in $V_1$ and $V_2$. Then $f^{-1}(W^2), f^{-1}(V_1^2)$ and $f^{-1}(V_2^2)$ are commensurable approximate subgroups by Lemma \ref{Lemma: Intersection of approximate subgroups}. But for $i=1,2$ we have $f^{-1}(W^2) \subset f^{-1}(V_i) \subset f^{-1}(V_i^2)$. So $f^{-1}(V_1)$ and $f^{-1}(V_2)$ are commensurable approximate subgroups.
\end{proof}

\begin{corollary}\label{Corollary: Definition of good models is independent of the choice of neighbourhood}
 Let $\L$ be an approximate subgroup of a (discrete) group $\Gamma$ and $f:\Gamma \rightarrow H$ be a good model of $(\L, \Gamma)$. If $U \subset H$ is a symmetric neighbourhood of the identity such that $f^{-1}(U) \subset \L$, then $f^{-1}(U)$ is an approximate subgroup commensurable to $\L$.
\end{corollary}

Admitting a good model is a property that is stable under group homomorphisms. 

\begin{lemma}\label{Lemma: Miscellaneous good models}
 Let $\L$ be an approximate subgroup of a group $\Gamma$. Suppose that $(\L,\Gamma)$ has a good model. We have :
 \begin{enumerate}
  \item if $\phi_1: \Gamma_1\rightarrow \Gamma$ is a group homomorphism, then $\phi_1^{-1}(\L)$ is an approximate subgroup and $(\phi_1^{-1}(\L),\Gamma_1)$ has a good model;
  \item if $\phi_2: \Gamma \rightarrow \Gamma_2 $ is a group homomorphism, then $(\phi_2(\L),\phi_2(\Gamma))$ has a good model.
  \end{enumerate}
 \end{lemma}

 \begin{proof}
  Let $f:\Gamma \rightarrow H$ be a good model of $(\L,\Gamma)$ and let $U \subset H$ be an open subset as in Definition \ref{Definition: Good models}. Set furthermore $\L_1:=\phi_1^{-1}(\L)$ and $f_1:=f \circ \phi_1$. Then $f_1(\L_1)=f(\L)$ is relatively compact and $f_1^{-1}(U) \subset \L_1$. Hence, $\L_1$ is an approximate subgroup by Lemma \ref{Lemma: Inverse images of compact neighbourhoods are pairwise commensurable} and $f_1$ is a good model of $(\L_1,\Gamma_1)$. Let us now prove (2). Take a good model $f: \Gamma \rightarrow H$ of $(\L,\Gamma)$ with dense image and $U \subset H$ a symmetric neighbourhood of the identity such that $f^{-1}(U^2) \subset \L$. Define $N:= \overline{f(\ker(\phi_2))}$ which is a normal subgroup since $f(\Gamma)$ is dense.  Now 
  $$(p_{H/N}\circ f)^{-1}(p_{H/N}(U)) \subset f^{-1}(U^2f(\ker(\phi_2))) \subset f^{-1}(U^2) \ker(\phi_2) \subset \L\ker(\phi_2)$$ where $p_{H/N}: H \rightarrow H/N$ denotes the natural projection. Therefore, the obvious map $\phi_2(\Gamma) \rightarrow H/N$ is a good model of $(\phi_2(\L),\phi_2(\Gamma))$. 
  \end{proof}

\subsection{Group-theoretic characterisation of good models} We will prove the following detailed version of Proposition \ref{Theorem: Characterisation of good models short version}:

\begin{theorem}\label{Theorem: Characterisation of good models}
 Let $\L$ be an approximate subgroup of a group $\Gamma$. The following are equivalent:
 \begin{enumerate}
 \item there is a good model $f:\Gamma \rightarrow H$ of $(\L, \Gamma)$; 
 \item there exists a sequence $(\L_n)_{n\geq 0}$ of approximate subgroups such that: 
        \begin{enumerate}
        \item $\L_0=\L$;
        \item for all integers $n\geq 0$ and all $\gamma \in \Gamma$, the approximate subgroups $\gamma\L_n\gamma^{-1}$ and $\L$ are commensurable;
        \item for all integers $n\geq 0$, we have $\L_{n+1}^2 \subset \L_n$; 
        \end{enumerate}
  \item there exists a family of subsets $\mathcal{B}$ such that:
        \begin{enumerate}
         \item there is $\Xi \in \mathcal{B}$ with $\Xi \subset \L$; 
         \item all elements of $\mathcal{B}$ contain $e$ and are commensurable to $\L$;
         \item for all $\L_1 \in \mathcal{B}$ and $\gamma \in \Gamma$, there is $\L_2 \in \mathcal{B}$ with $\gamma\L_2^{-1}\L_2 \gamma^{-1} \subset \L_1$.
        \end{enumerate}
  \end{enumerate}
  Moreover, when any of the three statements above is satisfied:
  \begin{itemize}
  \item[(4)] with $\mathcal{B}$ as in (3), we can choose a good model $f:\Gamma \rightarrow H$ such that $f$ has dense image and $\mathcal{B}$ is a neighbourhood basis for the identity with respect to the initial topology on $\Gamma$ given by $f$;
  \item[(5)] there is a good model $f_0:\Gamma \rightarrow H_0$ of $(\L, \Gamma)$ such that for any other good model $f:\Gamma \rightarrow H$ of $(\L, \Gamma)$ we have a continuous group homomorphism $\phi: H_0 \rightarrow H$ with compact kernel such that $f=\phi\circ f_0$;
 \item[(6)] if $\L$ is a $K$-approximate subgroup, then there exists a sequence $(\L_n)_{n\geq 0}$ with $\L_0=\L^8$ and as in (2) such that $\L$ is covered by $C_{K,n}$ left-translates of $\L_n$ for all $n \geq 0$, where $C_{K,n}$ is an integer that depends on $K$ and $n$ only.
 \end{itemize}
\end{theorem}

As mentioned in the introduction, Theorem \ref{Theorem: Characterisation of good models} is folklore and well-known to the expert. This is specially true in model theory and the more common approach goes through elementary model-theoretic tools (see \cite[Lemma 3.4]{MR2833482},\cite{MR3345797} and \cite[Lemma 6.6]{MR3090256} for a somewhat elementary approach). The first step is to embed $\L$ in a ``sufficiently saturated elementary extension'' $\underline{\Gamma}$ - for instance, by means of an ultra-power of $\L$ over a sufficiently large ultra-filter. Then one can obtain a good model by quotienting out a normal subgroup naturally associated with $(\Lambda_n)_{n \neq 0}$ and equip this quotient with the logic topology. This provides a sleek construction and highlights how extra structure on $\Lambda$ impacts the structure of the good model. 

Our goal here is to provide an elementary proof of Theorem \ref{Theorem: Characterisation of good models}. Indeed, it is already essentially contained in work of Weil on completion of uniform structures. Moreover, the form presented here will be more adapted to our discussion about regularity and continuity later on (Proposition \ref{Proposition: An amenable approximate subgroup has a good model}). We also hope that this will be of use to mathematicians outside model theory.

\begin{proof}
 $(1) \Rightarrow (2)$:
 
 Choose a neighbourhood of the identity $U \subset H$  such that $f^{-1}(U)\subset \L$. There exists a sequence $(U_n)_{n>0}$ of relatively compact symmetric neighbourhoods of the identity in $H$ such that $U_0=U$ and $U_{n+1}^2 \subset U_n$ for all integers $n\geq 0$. Define now $(\L_n)_{n\geq 0}$ by $\L_0=\L$ and $\L_n=f^{-1}(U_n)$. We readily check that for all integers $n \geq 0$ we have $\L_{n+1}^2 \subset \L_n$. Furthermore, for all $\gamma \in \Gamma$,  $\gamma\L_n\gamma^{-1}$ is an approximate subgroup commensurable to $\L$ by Corollary \ref{Corollary: Definition of good models is independent of the choice of neighbourhood}. So $(1) \Rightarrow (2)$ is proved.
 
 $(2) \Rightarrow (3)$:

 Let $(\L_n)_{n \geq 0}$ be as in $(2)$. For any two subsets $X,Y \subset G$ define
 $$ X^Y := \bigcap\limits_{y \in Y} yXy^{-1}.$$
 
 Define now $\mathcal{B}$ by 
 $$ \mathcal{B}:=\left\{\left(\L_n^2\right)^{F} \left| \ \forall n \in \N, \forall F \subset \Gamma, |F| < \infty\right.\right\}.$$
 
 We know that $\L_1^2 \subset \L$ and that for all $\Xi \in \mathcal{B}$ we have $e \in \Xi$ and $\Xi$ is an approximate subgroup commensurable to $\L$ (Lemma \ref{Lemma: Intersection of approximate subgroups}). 
 Now, for all $n \in \N$ and $F \subset \Gamma$ finite we have 
 $$ \left(\left(\L_{n+1}^2\right)^F\right)^2 \subset\left(\L_{n+1}^4\right)^F \subset \left(\L_n^2\right)^F$$
 and for $\gamma \in \Gamma$ we find
 $$ \gamma \left(\L_n^2\right)^F \gamma^{-1} \subset \left(\L_n^2\right)^{\gamma F}.$$
 So $\mathcal{B}$ satisfies (3).

$(3)\Rightarrow (1)$:

Equip the group $\langle \L\rangle$ with the topology defined by 
$$\mathcal{T} = \left\{U \subset \Gamma \vert \forall \gamma \in U, \exists \Xi \in \mathcal{B}, \gamma\Xi \subset U\right\}.$$
By \cite[Chapter III, \S 1.2, Proposition 1]{MR0358652} the topology $\mathcal{T}$ is the unique topology that makes $G$ into a topological group and such that $\mathcal{B}$ is a neighbourhood basis for $e$. Now, the closure $\overline{\{e\}}$ of the identity is a closed normal subgroup and the group $\Gamma/\overline{\{e\}}$ equipped with the quotient topology is the maximal Hausdorff factor of $\Gamma$. Let $p:\Gamma\rightarrow \Gamma/\overline{\{e\}}$ be the natural map. Then $\{p(\Xi)| \Xi \in \mathcal{B}\}$ is a neighbourhood basis for the identity in $\Gamma/\overline{\{e\}}$. But the subsets that belong to $\mathcal{B}$ are pairwise commensurable, so the neighbourhoods $\{p(\Xi)|\Xi \in \mathcal{B}\}$ are pre-compact. Hence, the topological group $\Gamma/\overline{\{e\}}$ has a completion by \cite[Theorem X]{weil1938espaces}. In other words, there is a locally compact group $H$ and a group homomorphism
 $$ i : \Gamma/\overline{\{e\}} \rightarrow H$$
 such that $i$ has dense image and is a homeomorphism onto its image. Define the map $f := i \circ p$. We will show that $f$ is a good model. The group $H$ is a complete space and $\L$ is pre-compact in the topology $\mathcal{T}$ according to assumption (b). So $f(\L)$ is a relatively compact subset of $H$. Recall that $i$ is a homeomorphism onto its image, the map $p$ is open and $\mathcal{B}$ is a neighbourhood basis for the identity. There is thus a neighbourhood of the identity $U \subset H$ such that $f^{-1}(U) \subset \L$ according to assumption (a).

 Statement (4) is straightforward from the proof of (3) $\Rightarrow$ (1). Let us prove (5). Let $\mathcal{T}_0$ denote the initial topology on $\Gamma$ with respect to the class of all good models $f: \Gamma \rightarrow H$ of $(\L,\Gamma)$. In other words, the topology $\mathcal{T}_0$ is generated by the family of subsets $\{f^{-1}(U)\}_{f,U}$ where $f: \Gamma \rightarrow H$ and $U$ run through all good models of $(\L,\Gamma)$ and all open subsets $U \subset H$. Define $\mathcal{B}_0$ as $\{ U \in \mathcal{T}_0 | e \in U \subset \L\}$. Since $\L$ is assumed to have a good model we know that $\mathcal{B}_0$ is not empty. So take $\Xi \in \mathcal{B}_0$. Then there are good models $(f_i:\Gamma \rightarrow H_i)_{1\leq i\leq r}$ of $(\L, \Gamma)$ and open relatively compact neighbourhoods of the identity $U_i \subset H_i$ for all $1 \leq i \leq r$ such that 
 $$\bigcap_{1\leq i \leq r} f_i^{-1}(U_i) \subset \Xi \subset \L.$$
But the map $f:=\prod_{1\leq i \leq r} f_i \Gamma \rightarrow H_i$ is a good model, so Corollary \ref{Corollary: Definition of good models is independent of the choice of neighbourhood}  implies that $\L$ is commensurable to $\bigcap_{1\leq i \leq r} f_i^{-1}(U_i)$, hence to $\Xi$. So $\mathcal{B}_0$ satisfies conditions (a) and (b) of (3). But conditions (c) and (d) of (3) are also satisfied since $(G,\mathcal{T}_0)$ is a topological group. Let now $f_0: \Gamma \rightarrow H_0$ be as in the proof of (3) $\Rightarrow$ (1). Then every good model $f: \Gamma \rightarrow H$ of $(\L,\Gamma)$ is continuous with respect to $\mathcal{T}_0$. According to the universal properties of quotients and completions (see \cite{weil1938espaces}) one can therefore find a continuous group homomorphism $\phi: H_0 \rightarrow H$ such that $\phi \circ f_0 = f$. 
 
 Take a good model $f: \Gamma \rightarrow H$ of $(\L,\Gamma)$ with dense image. We can find a relatively compact open symmetric neighbourhood of the identity $U \subset H$ such that $ \L \subset f^{-1}(U) \subset \L^2$. So $f^{-1}(U)$ is a $K^3$-approximate subgroup, and, hence, $U$ is a $K^3$-approximate subgroup as well. But by \cite[Cor. 5.6 and Lem. 5.4]{MR3438951} (or Lemma \ref{Lemma: Massicot-Wagner} below) we can find an symmetric open neighbourhood of the identity $S \subset U^4$ such that $S^{16}$ is contained in $U^4$ and $C_K$ left-translates of $S$ cover $U$ for some constant $C_K$ that depends on $K$ only. We thus define $\L_1= f^{-1}(S)$ and $C_{K,1}=C_K$. A proof by induction on $n$ then shows (6).
\end{proof}

\begin{proof}[Proof of Theorem \ref{Theorem: Characterisation of good models short version}.]
 Theorem \ref{Theorem: Characterisation of good models short version} is the equivalence ``(1) $\Leftrightarrow$ (2)'' in Theorem \ref{Theorem: Characterisation of good models}.
\end{proof}

\begin{remark}
We note here that, similarly, some results presented in section \ref{Section: Amenable} were foreshadowed by Weil's work on group topologies. We point for instance to the striking similarities between the work of Weil on completions of measurable groups and the method of Sanders and Croot--Sisask presented below.
\end{remark}

As a consequence of Theorem \ref{Theorem: Characterisation of good models} we show that Meyer subsets almost have a good model: 

\begin{proposition}\label{Proposition: Meyer subsets are contained in a model set}
Let $\L$ be an approximate subgroup of some group. If $\L$ is a Meyer subset then there is a positive integer such that $\L^n$ has a good model.  
\end{proposition}

\begin{proof}
 By Theorem \ref{Theorem: Characterisation of good models} there is a sequence $(\L_n)_{n\geq 0}$ of approximate subgroups commensurable to $\L$ such that $\L_{n+1}^2 \subset \L_n$ for all integers $n \geq 0$. Theorem \ref{Theorem: Characterisation of good models} again implies that the approximate subgroup $\L \cup \L_0$ has a good model, and so has $(\L \cup \L_0) \cap \langle\L\rangle$ by Lemma \ref{Lemma: Miscellaneous good models}. But $(\L \cup \L_0) \cap \langle\L\rangle$ is commensurable to $\L$ and $\langle\L\rangle$ is generated by $\L$. So there is a positive integer $n$ such that $(\L \cup \L_0) \cap \langle\L\rangle$ is contained in $\L^n$.
\end{proof}

This last result prompts the question: 

\begin{question}
With notations as in Proposition \ref{Proposition: Meyer subsets are contained in a model set}, can $n$ be chosen independently of $\Lambda$?
\end{question}

\subsection{Universal properties and compatibility with limits}
Approaching Theorem \ref{Theorem: Characterisation of good models} via saturated elementary extensions highlights that, to some extent, good models correspond to a certain notion of quotients of groups by approximate subgroups (see the remark below Theorem \ref{Theorem: Characterisation of good models} ). The elementary method presented above too enables us to build a good model that satisfies a quotient-like universal property. We think of this good model as a ``maximal'' or ``initial'' good model. We also identify a type of ``minimal'' or ``final'' good model. 
 
\begin{proposition}\label{Proposition: Extremal good models}
 Let $\L$ be an approximate subgroup of a group $G$ and let $\Gamma \subset G$ be a subgroup containing $\L$ such that $(\L, \Gamma)$ has a good model. We have:
 \begin{enumerate} 
 \item if $f_0: \Gamma \rightarrow H_0$ is as in part (5) of Theorem \ref{Theorem: Characterisation of good models}, then any group homomorphism $g: \Gamma \rightarrow L$ with target a topological group and such that $g(\L)$ is relatively compact factors through $f_0$ i.e. there exists a continuous group homomorphism $h: H_0 \rightarrow L$ such that $h \circ f_0 = g$;
 \item there is an approximate subgroup $\Xi \subset G$ commensurable to $\L$ and a good model $f: \langle\Xi\rangle\rightarrow H$ of $\Xi$ with dense image and target a connected Lie group without normal compact subgroup. Such a group $H$ is unique up to continuous isomorphisms. 
 \end{enumerate}
 \end{proposition}
 
 \begin{remark}
  One can note that $f_0$ enjoys a universal property similar to the Bohr compactification of a subgroup (i.e. the maximal group compactification of a given subgroup), and, indeed, if $\L=\Gamma$, then $H_0$ is the Bohr compactification of $\L$ . 
 \end{remark} 
 
 \begin{proof}  
 Take $\mathcal{B}_L$ a neighbourhood basis for the identity in $L$ made of symmetric subsets and take any $\mathcal{B}_{\L}$ as in part (3) of Theorem \ref{Theorem: Characterisation of good models}. Then as a consequence of Lemma \ref{Lemma: Intersection of commensurable subsets} we know that $\mathcal{B}:=\{ \Xi \cap g^{-1}(U) | \Xi \in \mathcal{B}_{\L}, U \in \mathcal{B}_L\}$ satisfies the assumptions of part (3) of Theorem \ref{Theorem: Characterisation of good models}. Now by part (4) of Theorem \ref{Theorem: Characterisation of good models} and by the universal properties of completions and quotients we can build a good model $f: \Gamma \rightarrow H$ such that $g$ factors through $f$. But according to part (5) of Theorem \ref{Theorem: Characterisation of good models} we know that $g$ factors through $f_0$ - proving (1).

  Take $f: \langle\L\rangle \rightarrow H$ a good model of $\L$ with dense image. By the Gleason--Yamabe theorem there are $O \subset H$ an open subgroup and a normal compact subgroup $K$ of $O$ such that $O/K$ is a connected Lie group without non-trivial normal compact subgroup. Then $\L':=f^{-1}(UK)$, where $U \subset O$ is any symmetric compact neighbourhood of the identity, is an approximate subgroup commensurable to $\L$ according to Lemma \ref{Lemma: Inverse images of compact neighbourhoods are pairwise commensurable}. But the composition of $f_{|\L'}$ and the natural projection $O\rightarrow O/L$ is easily checked to be a good model of $\L'$. Now take $\Xi$ and $\Xi'$ approximate subgroups commensurable to $\L$, and let $f: \langle\Xi\rangle \rightarrow H$ and $f':\langle\Xi'\rangle \rightarrow H'$ be good models of $\Xi$ and $\Xi'$ respectively. Suppose moreover that both satisfy (2). Let $D: \langle\Xi^2 \cap \Xi'^2\rangle \rightarrow H \times H'$ be the diagonal map and let $\Delta$ be the closure of its image. We want to show that $\Delta \cap H$ and $\Delta \cap H'$ are compact subgroups. Let $U$ denote a relatively compact open neighbourhood of the identity in $\Delta$. Then $D(\langle\Xi^2 \cap \Xi'^2\rangle) \cap U(\Delta \cap H')$ is dense in $U(\Delta \cap H')$. But the projection of $U(\Delta \cap H')$ to $H$ is a relatively compact set $U_H$. So $D^{-1}(U(\Delta \cap H'))$ is contained in $f^{-1}(U_H)$ which is covered by finitely many translates of $\Xi$ (Corollary \ref{Corollary: Definition of good models is independent of the choice of neighbourhood}). So $U(\Delta \cap H') \subset \overline{D(f^{-1}(U_H))}$ which is compact, thus proving that $\Delta \cap H'$ is compact. A symmetric argument shows that $\Delta \cap H$ is compact. Moreover, $\Xi^2 \cap \Xi'^2$ is commensurable to both $\Xi$ and $\Xi'$ by Lemma \ref{Lemma: Intersection of commensurable subsets}. So the projection of $\overline{\Xi^2 \cap \Xi'^2}$ to $H$ contains an open subset by the Baire category theorem. Therefore, $\Delta$ projects surjectively to $H$ by connectedness. Likewise the projection of $\Delta$ to $H'$ is surjective. So $\Delta \cap H$ and $\Delta \cap H'$ are compact normal subgroups of $H$ and $H'$ respectively. Which means that they are trivial. As a consequence, $\Delta$ is the graph of a continuous isomorphism $\phi: H \rightarrow H'$ such that $f'_{|\langle\Xi^2 \cap \Xi'^2\rangle} = \phi \circ f_{|\langle\Xi^2 \cap \Xi'^2\rangle}$. This proves (2).
 \end{proof}

 We will show later on that, upon dropping the requirement that $f_0$ be a good model, part (1) of Proposition \ref{Proposition: Extremal good models} can be extended to all approximate subgroups. See Proposition \ref{Proposition: Bohr compactification} below.

 A key point in both \cite{MR3090256} and \cite{MR2833482} consists in utilising ultraproducts of approximate subgroups to study all members of a family of approximate subgroups at once. Notably, they showed that ultraproducts of finite approximate subgroups have a good model, as this enabled them to use features of Lie groups and locally compact groups (such as tools developed by Gleason, Yamabe and others in the resolution of Hilbert's fifth problem) to tackle problems about finite approximate subgroups. In a similar fashion, ultraproducts of locally compact approximate subgroups (i.e. compact symmetric neighbourhoods of the identity) were shown to have a good model in \cite{MR3438951} in order to obtain uniform and quantitative - although non-effective - versions of the Gleason--Yamabe structure theorem for locally compact groups. We show that in general the property \emph{to have a good model} is stable under ultrapoducts.
 
\begin{proposition}\label{Proposition: Limits and good models}
 Let $(\L_i)_{i \in I}$ be a family of $K$-approximate subgroups of the groups $(\Gamma_i)_{i \in I}$ for some fixed integer $K$.
 \begin{enumerate} 
 
   \item(\cite[App. A]{MR3267520}) the ultraproduct $\underline{\L}:=\prod_{i\in I} \L_i / \mathcal{U}$ is an approximate subgroup for any ultrafilter $\mathcal{U}$ over $I$;
   \item if moreover $\L_i$ has a good model for all $i \in I$, then $\underline{\L}^8$ has a good model.
 \end{enumerate}
 Suppose that there are a directed order $\leq$ on $I$ and injective group homomorphisms $\psi_{ij}: \Gamma_i \rightarrow \Gamma_j$ for all $i \leq j$ such that $\psi_{jk} \circ \psi_{ij} = \psi_{ik}$ for all $i \leq j \leq k$. And suppose moreover that $\psi_{ij}(\L_i)\subset \L_j$ whenever $i \leq j$. Then:
 \begin{itemize}
  \item[(3)] the direct limit $\varinjlim_{I} \L_i^2$ is an approximate subgroup;
  \item[(4)] if moreover $\L_i$ has a good model for all $i \in I$, then $\varinjlim_{I} \L_i^8$ has a good model.
 \end{itemize}
\end{proposition}

\begin{proof}
 If $\L_i$ has a good model for all $i \in I$ then there are constants $(C_{K,n})_{n\geq 0}$ and sequences $(\L_{i,n})_{n\geq 0}$ as in part (6) of Theorem \ref{Theorem: Characterisation of good models}. But then for any ultrafilter $\mathcal{U}$ over $I$ we know that $\underline{\L}^8$ is covered by $C_{K,n}$ left-translates of $\underline{\L_n}:= \prod_{i\in I} \L_{i,n}/ \mathcal{U}$ for all $n \geq 0$ (see e.g. \cite[App. A]{MR3267520} for background material on ultraproducts of groups). So $(\underline{\L_n})_{n\geq 0}$  satisfies part (2) of Theorem \ref{Theorem: Characterisation of good models} and $\underline{\L_0}=\underline{\L}^8$ has a good model.
 
 Let $\mathcal{U}$ be an ultrafilter on $I$ that contains the cofinal subsets $\{i \in I | j \leq i \}$ for all $j \in I$. Write $\underline{\L}$ and $\underline{\Gamma}$ for the ultraproducts of $(\L_i)_{i \in I}$ and $(\Gamma_i)_{i \in I}$ over $\mathcal{U}$ respectively. By the universal property of direct limits there is a natural map $\phi: \varinjlim_I \Gamma_i \rightarrow \underline{\Gamma}$ and we compute that $\phi^{-1}(\underline{\L})=\varinjlim_I \L_i$. So $\varinjlim_{I} \L_i^2$ is an approximate subgroup by Lemma \ref{Lemma: Pull-back of commensurable approximate subgroups}. If every $\L_i$ has a good model, then the approximate subgroup $\underline{\L}^8$ has a good model by  (2). So $\varinjlim_{ I} \L_i^8$ has a good model by Lemma \ref{Lemma: Miscellaneous good models}. 
\end{proof}

 One may wonder if a converse to part (2) (and (4)) of Proposition \ref{Proposition: Limits and good models} holds. The next section will give rise to an interesting counter-example. Namely:

\begin{lemma}\label{Lemma: Ultraproduct of approximate subgroups without good models}
  There is a sequence $(\L_n)_{n\geq 0}$ of approximate subgroups such that $\prod_{n \geq 0} \L_n / \mathcal{U}$ has a good model but for all $n \geq 0$, $\L_n$ is not a Meyer subset, where $\mathcal{U}$ is any non-principal ultrafilter on $\N$. 
 \end{lemma}
 
 We delay the proof to the end of the next subsection.

\subsection{An approximate subgroup without a good model}
Recall that a \emph{quasi-morphism} of a group $G$ is a map $f: G \rightarrow \R$ such that
$$ C(f):=\sup_{g_1,g_2 \in G} |f(g_1g_2) - f(g_1) - f(g_2)| < \infty.$$
We say that $f$ is \emph{symmetric} if for all $g\in G$ we have $f(g^{-1})=-f(g)$ and that it is \emph{homogeneous} if for all $n\in \Z$ and $g \in G$ we have $f(g^n)=nf(g)$. Just like group homomorphisms quasi-morphisms give rise to families of approximate subgroups. The approximate subgroups produced that way are often called \emph{quasi-kernels}.

\begin{lemma}\label{Lemma: A quasi-kernel is an approximate subgroup}
 Let $G$ be a group and $f: G \rightarrow \R$ be a symmetric quasi-morphism. Then for all $R > C(f)$ the set $f^{-1}(\left[-R;R\right])$ is a $2\frac{2R + C(f)}{R- C(f)}+1$ -approximate subgroup.
\end{lemma}

\begin{proof}
 Let $\L$ denote the set $f^{-1}(\left[-R;R\right])$. Then $\L$ is symmetric since $f$ is symmetric and the set $f(\L^2)$ is contained in $[-2R-C(f);2R + C(f)]$. Set $\delta:=R-C(f)> 0$ and choose a finite subset $F \subset \L^2$ such that $f(F)$ is a maximal $\delta$-separated subset of $f(\L^2)$. We know that $|F| < 2\frac{2R + C(f)}{\delta}+1$ and in addition we have 
 $$f(\L^2) \subset \bigcup\limits_{\gamma \in F} f(\gamma) + [-\delta;\delta].$$ 
 Take $\lambda \in \L^2$ and $\gamma \in F$ such that $\vert f(\lambda)-f(\gamma) \vert\leq \delta$. We have  
 $$ \vert f(\gamma^{-1}\lambda)- \left(f(\lambda)-f(\gamma)\right)\vert \leq C(f),$$
 so
 $$\vert f(\gamma^{-1}\lambda)\vert \leq C(f)+\delta = R.$$
 Hence, $\gamma^{-1}\lambda \in \L$ and $\L^2 \subset F\L$. 
\end{proof}

Since all bounded maps are quasi-morphisms, quasi-morphisms are often studied up to a bounded error. This gives an equivalence relation between quasi-morphisms that can be translated as a commensurability condition on quasi-kernels. 

\begin{lemma}\label{Lemma: Equivalent quasi-morphisms give commensurable quasi-kernels}
 Let $G$ be a group and $f_1,f_2: G \rightarrow \R$ be two symmetric quasi-morphisms. Suppose that $\eta:=\sup_{g \in G}\vert f_1(g)-f_2(g)\vert < \infty$. Then for all $R_1 > C(f_1)$ and $R_2 > C(f_2)$ the approximate subgroups $\L_1:=f_1^{-1}(\left[-R_1;R_1\right])$ and $\L_2:=f_2^{-1}(\left[-R_2;R_2\right])$ are commensurable. More precisely, there is $F \subset G$ with $|F| \leq \max( \frac{2(R_1 + \eta)}{R_2 - C(f_2)}+1, \frac{2(R_2+\eta)}{R_1 - C(f_1)}+1)$ such that $\L_1 \subset F \L_2$ and $\L_2 \subset F\L_1$.  
 
 Conversely, if $\Lambda_1$ and $\Lambda_2$ are commensurable, then there is $\alpha \in \mathbb{R}$ non-trivial such that $f_1 - \alpha f_2$ is bounded. 
\end{lemma}

\begin{proof}
 Write $\delta_1:=R_1 - C(f_1)$. Choose a maximal $\delta_1$-separated subset $F_1$ of $f_1(\L_2)$. Since $f_1(\L_2) \subset [-(R_2+\eta);R_2 + \eta]$ we know that $|F_1| \leq \frac{2(R_2+\eta)}{\delta_1}+1$. As in the proof of Lemma \ref{Lemma: A quasi-kernel is an approximate subgroup} we find that $\L_2 \subset F_1\L_1$. By symmetry there is $F_2 \subset G$ with $|F_2| \leq \frac{2(R_1 + \eta)}{R_2 - C(f_2)}+1$ such that $\L_1 \subset F_2 \L_2$.

Let us now prove the converse statement. Both $f_1$ and $f_2$ are within bounded distance of an homogeneous quasi-morphism. Therefore, by the first part of Lemma \ref{Lemma: Equivalent quasi-morphisms give commensurable quasi-kernels} we only have to prove the converse assuming that $f_1$ and $f_2$ are homogeneous. First of all, if $f_1 = 0$, then $\Lambda_2 = f_2^{-1}(\left[-R_2;R_2\right])$ is commensurable to $G$. So $f_2(G)$ is bounded. But $f_2$ is homogeneous, so $f_2=0$. Suppose now that $f_1$ is non-trivial and take $g_0 \in G$ such that $f_1(g_0)> 0$. Define that map
 $$\hat{f}: g \mapsto f_1(g_0)f_2(g) - f_2(g_0)f_1(g).$$
 It is a homogeneous quasi-morphism with $\hat{f}(g_0)=0$. Moreover, any set commensurable to $\L_1$ (equivalently, to $\Lambda_2$) has bounded image by $\hat{f}$. For all $g \in G$, there is $n \in \Z$ such that $\vert f_1(g) - nf_1(g_0)\vert \leq |f_1(g_0)|$. Thus, $$G= \langle g_0 \rangle f_1^{-1}([-C(f_1)-f_1(g_0);f_1(g_0)+C(f_1)]).$$ But  $f_1^{-1}([-C(f_1)-f_1(g_0);f_1(g_0)+C(f_1)])$ is commensurable to $\L_1$ according to the first part of the proof and, hence, is mapped to a bounded set by $\hat{f}$. So $\hat{f}$ must have bounded image and, therefore, must be trivial. In other words $f_2=\frac{f_2(g_0)}{f_1(g_0)}f_1$.
\end{proof}

Our main result links properties of quasi-morphisms to whether the quasi-kernel is a Meyer subset or not.

\begin{proposition}\label{Proposition: A quasi-kernel has a good model if and only if it comes from a homomorphism}
Let $G$ be a finitely generated group and let $f: G \rightarrow \R$ be a homogeneous quasi-morphism. Choose a real number $R > C(f)$. If the approximate subgroup $f^{-1}(\left[-R;R\right])$ is a Meyer subset then $f$ is a group homomorphism. 
\end{proposition}

\begin{proof}
 If $f$ is bounded, then $f=0$. So assume that $f$ is unbounded. Take $R' > C(f)$ such that $f^{-1}([-R';R'])$ generates $G$. We know that $f^{-1}([-R';R'])$ is an approximate subgroup (Lemma \ref{Lemma: A quasi-kernel is an approximate subgroup}) and a Meyer subset (Lemma \ref{Lemma: Equivalent quasi-morphisms give commensurable quasi-kernels}). By Proposition \ref{Proposition: Meyer subsets are contained in a model set} there is an integer $n \geq 1$ such that there are a good model $f_0:G \rightarrow H$ of some power, say $n$, of $f^{-1}([-R';R'])$ with dense image. In particular, $f_0$ is a good model of $\L:=f^{-1}([-n(R'+C(f));n(R'+C(f))])$ according to Lemma \ref{Lemma: Equivalent quasi-morphisms give commensurable quasi-kernels}. Since $f_0(G)$ is dense in $H$ we have that $\overline{f_0(\L)}$ is a neighbourhood of the identity. So the subgroup generated by the compact set $\overline{f_0(\L)}$ is open and contains $f_0(G)$, hence equals $H$. The group $H$ is thus compactly generated. We will now show that $f=cf_0$ for some real number $c > 0$. We start with two claims:
 
 \begin{claim}\label{Claim: Set of commutators is relatively compact}
  The set of commutators $\{h_1h_2h_1^{-1}h_2^{-1} \vert h_1,h_2 \in H \}$ is relatively compact.
 \end{claim}
 
 By Lemma \ref{Lemma: Equivalent quasi-morphisms give commensurable quasi-kernels} the approximate subgroup $f_0(f^{-1}([-3C(f);3C(f)]))$ is commensurable to $f_0(\L)$. So $K := \overline{f_0(f^{-1}([-3C(f);3C(f)]))}$ is compact. Take now $\gamma_1,\gamma_2 \in G$. We have $|f(\gamma_1\gamma_2\gamma_1^{-1}\gamma_2^{-1})| \leq 3C(f)$, hence $f_0(\gamma_1\gamma_2\gamma_1^{-1}\gamma_2^{-1}) \in K$. But $f_0(G)$ is dense in $H$ so we find $\{h_1h_2h_1^{-1}h_2^{-1} \vert h_1,h_2 \in H \} \subset K$.
 
 \begin{claim}\label{Claim: There is only one geodesic up to coarse equivalence}
  Let $\gamma_0 \in G$ be such that $f(\gamma_0) > 0$. Then there is a compact subset $K \subset H$ such that for all $\gamma \in G$ we have $f_0(\gamma) \subset \langle f_0(\gamma_0)\rangle K$.
 \end{claim}

 Write $R_0:= f(\gamma_0)$, then the subset $f^{-1}([-R_0-C(f);R_0+C(f)])$ is commensurable to $\L$ by Lemma \ref{Lemma: Equivalent quasi-morphisms give commensurable quasi-kernels}. So the subset 
 $$K:=\overline{f_0(f^{-1}([-R_0-C(f);R_0+C(f)]))}$$ is compact. But for any $\gamma \in G$ there is an integer $l$ such that $\vert f(\gamma)-lf(\gamma_0) \vert \leq R_0$ so $\gamma_0^{-l}\gamma\in K$.

 Now all conjugacy classes of $H$ are relatively compact according to Claim \ref{Claim: Set of commutators is relatively compact}. So we can find a compact normal subgroup $K \subset H$ and non-negative integers $k,l$ such that $H/K \simeq \R^k \times \Z^l$ (see \cite{MR0165031}). We will show that $k+l \leq 1$. Let $p: H \rightarrow H/K$ denote the natural projection. Then the group homomorphism $p\circ f_0$ has dense image and $p\circ f_0(\L)$ is relatively compact. If $H/K$ is compact then $H/K \simeq \{e\}$. Otherwise we can find $\gamma_0 \in G\setminus \L$ so $f(\gamma_0)\geq R > 0$. According to Claim \ref{Claim: There is only one geodesic up to coarse equivalence} every $\gamma \in G$ satisfies $p\circ f_0(\gamma) \in \langle p\circ f_0(\gamma_0)\rangle L$ where $L$ is some compact subset of $H/K$. Therefore $\langle p\circ f_0(\gamma_0)\rangle$ is an infinite cyclic co-compact subgroup, and, hence, $k+l \leq 1$. Choose a neighbourhood  $U\subset H$ of the identity such that $f_0^{-1}(U) \subset \L$. Since $K$ is a compact subgroup of $H$ and the subgroup $f_0(G)$ is dense we can find an integer $m \geq 0$ such that $K \subset f(\L^m)U$. Then $V:=p(U) \subset H/K$ is such that
  $$f_0^{-1}(p^{-1}(V))\subset f_0^{-1}(UK) \subset \L^{m+2}.$$ So $p\circ f_0$ is a good model of $\L^{m+2}$ with image dense in $\R$ or $\Z$. By the converse of Lemma \ref{Lemma: Equivalent quasi-morphisms give commensurable quasi-kernels} $p \circ f_0$ and $f$ are within bounded distance of each other. A contradiction. 
\end{proof}

\begin{proof}[Proof of Theorem \ref{Theorem: Example of approximate subgroup without a good model}.]

Recall that $w$ is a reduced word of length $l$ in $F_2$ the free group over $\{a,b\}$  and that $o(g,w)$ counts the occurrence of $w$ as a reduced sub-word of $g$ with overlap. Suppose that $w \notin \{a, b, a^{-1},b^{-1},e\}$. According to \cite[\S 3.(a)]{brooks1981some} the map 
\begin{align*}
 f_w: & F_2 \longrightarrow \R \\
      & g \longmapsto o(g,w)-o(g,w^{-1})
\end{align*}
is a symmetric quasi-morphism with $C(f_w) \leq 3(l -1)$. Moreover, $f_w$ is within distance $\delta$ of a unique homogeneous quasi-morphism, $\tilde{f}_w$ say, that is not a group homomorphism. But according to Lemma \ref{Lemma: A quasi-kernel is an approximate subgroup} the set $$\left\{ g \in F_2  \left| |o(g,w)-o(g,w^{-1})| \leq 3l\right. \right\} = f_w^{-1}([-3l;3l)])$$ is an approximate subgroup. Moreover, by Lemma \ref{Lemma: Equivalent quasi-morphisms give commensurable quasi-kernels} it is commensurable to $\tilde{f}_w^{-1}([-C(\tilde{f}_w)-\delta;C(\tilde{f}_w)+\delta])$. So if $\left\{ g \in F_2  \left| |o(g,w)-o(g,w^{-1})| \leq 3l \right.\right\}$ is a Meyer subset, then $\tilde{f}_w^{-1}([-C(\tilde{f}_w)-\delta;C(\tilde{f}_w)+\delta])$ is a Meyer subset, and, hence, $\tilde{f}_w$ is a group homomorphism according to Proposition \ref{Proposition: A quasi-kernel has a good model if and only if it comes from a homomorphism}. A contradiction. 
\end{proof}

Theorem \ref{Theorem: Example of approximate subgroup without a good model} seems to contradict that ``even without the definable amenability assumption a suitable Lie model exists'' conjectured in \cite[p. 57]{MR3345797}. It is interesting to note that another counter-example to this conjecture is given in  \cite{hrushovski2019amenability} and that it is also built thanks to Brooks' quasi-morphisms. In contrast, we note that Hrushovski provides in \cite{hrushovski2020beyond} a partial positive answer to the conjecture, showing that one can always construct such a model using a combination both homomorphisms and quasi-homomorphisms. 

 Finally, we give a proof of Lemma \ref{Lemma: Ultraproduct of approximate subgroups without good models}:

 \begin{proof}[Proof of Lemma \ref{Lemma: Ultraproduct of approximate subgroups without good models}.]
  Let $w$ be a reduced word of length $l$ in $F_2$ the free group over $\{a,b\}$ and suppose that $w \notin \{a, b, a^{-1},b^{-1},e\}$. Let $f_w$ be as in the proof of Theorem \ref{Theorem: Example of approximate subgroup without a good model}. Define $\L_n:=f_w^{-1}([-3ln;3ln])$ for all $n \geq 0$. Then $\L_n^2 \subset \L_{3n}$ and $K$ left translates of $\L_n$ cover $\L_{3n}$ for some $K$ independent of $n$ (Lemma \ref{Lemma: Equivalent quasi-morphisms give commensurable quasi-kernels}). So the sequence of subsets $(\underline{\L}_k)_{k \geq 0}$ defined by $\underline{\L}_k:=\prod_{n \geq 0} f_w^{-1}([-3ln/3^k;3ln/3^k])$ is made of well-defined approximate subgroups commensurable to $\prod_{n \geq 0} \L_n / \mathcal{U}$ (see e.g. \cite[App. A]{MR3267520}). Besides $\underline{\L}_k^2 \subset \underline{\L}_{k-1}$ for all $k \geq 1$ since $\mathcal{U}$ is non-principal. So $\prod_{n \geq 0} \L_n / \mathcal{U}$ has a good model according to Theorem \ref{Theorem: Characterisation of good models}. However, for all $n \geq 0$, $\L_n$ is not a Meyer subset  according to Theorem \ref{Theorem: Example of approximate subgroup without a good model}. 
 \end{proof}

\subsection{Good models and cut-and-project schemes}

In this section we relate good models (Definition \ref{Definition: Good models}) to the non-commutative cut-and-project schemes (Definition \ref{Definition: Cut-and-project scheme}). Note that when the ambient group is abelian, Meyer was the first to notice the striking link between cut-and-project schemes and some large approximate subgroups (see \cite{meyer1972algebraic} and \cite{schreiber1973approximations} for this and more).  
 
\begin{lemma}\label{Lemma: The graph of a discrete approximate subgroup is discrete}
 Let $\L$ be a discrete approximate subgroup of a locally compact group $G$ and let $\Gamma$ be a group that contains it. If $(\L, \Gamma)$ has a good model $f: \Gamma \rightarrow H$, then the graph of $f$ defined by $\Gamma_f:=\{(\gamma,f(\gamma))\vert \gamma \in \Gamma\}$ is a discrete subgroup of $G \times H$.
\end{lemma}

\begin{proof}
 Choose a neighbourhood of the identity $U \subset H$ such that $f^{-1}(U) \subset \L$ (Definition \ref{Definition: Good models}) and an open subset $V \subset G$ such that $V \cap \L =\{e\}$. For $\gamma \in \Gamma$ we know that $(\gamma,f(\gamma)) \in V \times U$ implies $f(\gamma) \in U$, hence $\gamma \in \L$. But $\gamma \in V$, so $\gamma=e$ and we find $\Gamma_f \cap \left(V \times U \right)= \{e\}$. 
\end{proof}

An easy consequence of Lemma \ref{Lemma: The graph of a discrete approximate subgroup is discrete}, in the spirit of \cite[Proposition 2.3, (iv)]{bjorklund2016approximate}, asserts that the graph of a good model of an approximate subgroup $\L$ has finite co-volume as soon as $\L$ itself has finite co-volume. Namely:

\begin{proposition}\label{Proposition: Weak model sets with finite co-volume are model sets}
 Let $\L$ be an approximate lattice of a locally compact group $G$. Suppose that $\L$ has a good model $f: \Gamma \rightarrow H$ with dense image. The graph $\Gamma_f$ of $f$ is a lattice in $G \times H$ and $\L$ is contained in and commensurable to a model set. 
\end{proposition}

\begin{proof} Define $W_0:=\overline{f(\L)}$. Then $\Gamma_f$ is discrete by Lemma \ref{Lemma: The graph of a discrete approximate subgroup is discrete}. Moreover, we know that for all $(g,h) \in G \times H$ there is $\gamma_1 \in \Gamma_f$ such that $(g,h)\gamma_1^{-1} \in G \times W_0$. By assumption we have $\gamma_2 \in \Gamma_f \cap (G \times W_0)$ such that $p_G((g,h)\gamma_1^{-1}\gamma_2^{-1}) \in \mathcal{F}$. Therefore, we know that $(g,h) \in \left(\mathcal{F} \times W_0W_0^{-1}\right)\Gamma_f$. So $\Gamma_f$ is a lattice in $G \times H$. Hence, $\L$ is contained in and commensurable to the model set $\L\ker(f) = P_0(G,H,\Gamma_f, W_0)$.
\end{proof}

\begin{proof}[Proof of Proposition \ref{Proposition: An approximate lattice has a good model if and only if it is a model set}.]
 Assume that $\L$ is a model set. Let $(G,H,\Gamma)$ be a cut-and-project scheme and let $W_0$ be a neighbourhood of the identity $W_0 \subset H$ such that $P_0(G,H,\Gamma, W_0)=\L$ (Definition \ref{Definition: Model sets and Meyer sets}). Denote by $p_G: G \times H \rightarrow G$ and $p_H: G \times H \rightarrow H$ the natural projections. The map $\left(p_G\right)_{\vert \Gamma}$ is injective and $$\L=P_0(G,H,\Gamma,W_0)= p_G\left((G \times W_0) \cap \Gamma\right).$$ But $$p_G\left((G \times W_0) \cap \Gamma\right) = \left(p_H\circ\left(p_G\right)_{\vert \Gamma}^{-1} \right)^{-1}(W_0)$$ where we think of $\left(p_G\right)_{\vert \Gamma}$ as a bijective map from $\Gamma$ to $p_G(\Gamma)$. We know that $\langle\L\rangle \subset p_G(\Gamma)$, so set 
 \begin{align*}
  \tau :  \langle\L\rangle &\longrightarrow H \\
          \gamma &\longmapsto p_H\circ\left( p_G\right)_{\vert \Gamma}^{-1}(\gamma).
 \end{align*}
Then $\tau$ is a group homomorphism, $W_0$ is a symmetric relatively compact neighbourhood of the identity and $\tau^{-1}(W_0) = \L$. So $\tau$ is a good model of $\L$.

Conversely, $\L$ has a good model so it is contained in and commensurable to a model set by Proposition \ref{Proposition: Weak model sets with finite co-volume are model sets}. 

\end{proof}

\begin{remark}
 Note that the map $\tau$ introduced in the first part of the proof of Proposition \ref{Proposition: An approximate lattice has a good model if and only if it is a model set} is well-known in the abelian setting and is called the \emph{star-map} (see for instance \cite[\S 7.2]{MR3136260}).
\end{remark}

\section{A closed-approximate-subgroup theorem}\label{Section: A closed-approximate-subgroup theorem}
We give in this section a proof of Theorem \ref{Theorem: Closed-approximate-subgroup theorem Lie group form} and investigate some applications. 
\subsection{Globalisation in Hausdorff Topological Groups}
We start by proving a general form of Theorem \ref{Theorem: Closed-approximate-subgroup theorem Lie group form}. 

\begin{theorem}\label{Theorem: Closed-approximate-subgroup theorem general form}
 Let $\L$ be a compact approximate subgroup of a Hausdorff topological group $G$ and $\Gamma$ a subgroup that contains $\Lambda$ and commensurates it. There is a locally compact group $H$, an injective continuous group homomorphism $\phi:H \rightarrow G$ and a compact symmetric neighbourhood of the identity $V$ that generates $H$ such that $\phi(V)=\L^2$ and $\phi(H)=\Gamma$. 
\end{theorem}

The key observation needed to prove Theorem \ref{Theorem: Closed-approximate-subgroup theorem general form} is the fact that locally a closed approximate subgroup behaves like a group.

\begin{lemma}\label{Lemma: Closed approximate subgroups are locally stable by product}
 Let $\L$ be a closed approximate subgroup of a locally compact group $G$ and let $\Xi$ be a subset covered by finitely many left-translates of $\L$. There is a neighbourhood of the identity $U(\Xi) \subset G$ such that :
 $$ \Xi \cap U(\Xi) \subset \L^2 \cap U(\Xi).$$
\end{lemma}

\begin{proof}
 Choose a finite subset $F \subset G$ such that $\Xi \subset F\L$. Define the open subset  
 $$U(\Xi):=G\setminus\left(\bigcup\limits_{f \in F, f\notin \L}f\L\right).$$ Since $e \in f\L$ implies $f\in \L^{-1}=\L$, the subset $U(\Xi)$ contains the identity. We thus have 
 \begin{align*}
  U(\Xi) \cap \Xi & \subset U(\Xi)\cap F\L \\
                 & \subset \bigcup\limits_{f\in F, f \in \L}f \L \\
                 & \subset \L^2.
 \end{align*}
\end{proof}

Lemma \ref{Lemma: Closed approximate subgroups are locally stable by product} asserts that the restriction of the group operations of $G$ to some neighbourhood of the identity in $\L$ gives rise to a structure of a \emph{local topological group}. But it is well-known that a local topological group $\Omega$ with a continuous embedding into a global group $H$ can be globalised, and that the procedure gives rise to a topological group structure on the subgroup of $H$ generated by the image of $\Omega$ (see \cite{MR1406004,MR2680491} for this and more on local groups). The language of good models, and Theorem \ref{Theorem: Characterisation of good models} in particular, can be used to achieve the same result for the whole of $\L$ (more precisely the pair $(\L, \Gamma)$) instead of an open subset of $\L$ - thus proving Theorem \ref{Theorem: Closed-approximate-subgroup theorem general form}. 

\begin{proof}[Proof of Theorem \ref{Theorem: Closed-approximate-subgroup theorem general form}.]
For all $\gamma \in \Gamma$ let $U(\gamma)$ be the neighbourhood of the identity such that $\gamma\L^4\gamma^{-1} \cap U(\gamma) \subset \L^2$ (Lemma \ref{Lemma: Closed approximate subgroups are locally stable by product}). Choose a neighbourhood basis for the identity $\mathcal{B}$ made of closed subsets in $G$ and define $\mathcal{B}_{\L}$ as the family of subsets $\{ \L^2 \cap U^{-1}U | U \in \mathcal{B}\}$. The subsets in $\mathcal{B}_{\L}$  are all contained in and commensurable to $\L^2$ by Lemma \ref{Lemma: Intersection of commensurable subsets}.  Take any $U \in\mathcal{B}$ and any $\gamma \in \Gamma$ and choose $V \in \mathcal{B}$ such that $\gamma (V^{-1}V)^2\gamma^{-1} \subset U(\gamma)\cap U$. Then 
\begin{align*}
 \gamma \left(V^{-1}V  \cap \L^2\right)^2 \gamma^{-1} & \subset \gamma(V^{-1}V)^2\gamma^{-1} \cap \gamma\L^4\gamma^{-1} \\
                                                      & \subset U(\gamma)\cap U \cap \gamma\L^4\gamma^{-1} \\
                                                      & \subset U \cap \L^2.
\end{align*}
So $\mathcal{B}_{\L}$ checks all conditions of (3) of Theorem \ref{Theorem: Characterisation of good models}. 

 Choose a good model $f: \Gamma \rightarrow H$ of $(\L^2, \Gamma)$ with dense image and such that $\mathcal{B}_{\L}$ is a neighbourhood basis for the identity in the initial topology ((4) of Theorem \ref{Theorem: Characterisation of good models}).  We know the family $\{\overline{f(\Xi)}| \Xi \in \mathcal{B}_{\L} \}$ is a neighbourhood basis for the identity in $H$ and $\ker(f) = \bigcap_{\Xi \in \mathcal{B}_{\L}} \Xi \subset \bigcap_{U \in \mathcal{B}}U^{-1}U = \{e\}$. Take $\lambda \in \L^{2}$ and take any $U \in \mathcal{B}$ with $U^{-1}U \subset U(e)$. Then 
 $$ \L^2 \cap \lambda U^{-1}U = \lambda (\lambda^{-1}\L^2 \cap U^{-1}U) \subset \lambda (\L^4 \cap U^{-1}U) = \lambda (\L^2 \cap U^{-1}U), $$
 so the restriction of $f$ to $\L^2$ is a continuous map. Hence,  $\{f(\Xi)| \Xi \in \mathcal{B}_{\L} \}=\{\overline{f(\Xi)}| \Xi \in \mathcal{B}_{\L} \}$ is a neighbourhood basis for the identity in $H$. Since $f(\L^2)$ has non-empty interior and $f(\Gamma)$ is dense in $H$, we find $f(\Gamma)=f(\Gamma)f(\L^2)=H$. So $f$ is bijective. But for all $\Xi \in \mathcal{B}_{\L}$ we have $f^{-1}(f(\Xi)) = \Xi$. So $f^{-1}: H \rightarrow G$ is a continuous one-to-one group homomorphism and $\L^2$ is the image of the compact neighbourhood of the identity $f(\L^2)$. 
\end{proof}

We can now turn to the proof of Theorem \ref{Theorem: Closed-approximate-subgroup theorem Lie group form}. This result is akin to Cartan's closed-subgroup theorem as it shows that closed approximate subgroups of Lie groups have a Lie group structure, at least locally. In particular, it enables one to define unambiguously the Lie algebra associated to a closed approximate subgroup of a Lie group. Note that this last fact could also be proved as a consequence of Lemma \ref{Lemma: Closed approximate subgroups are locally stable by product} and \cite[Chapter III, \S 8, Prop. 2]{BourbakiNicolas1975LgaL}. 

\begin{proof}[Proof of Theorem \ref{Theorem: Closed-approximate-subgroup theorem Lie group form}]
Apply Theorem \ref{Theorem: Closed-approximate-subgroup theorem general form} to $(\L^2 \cap V^2 , \langle\L\rangle)$ where $V$ is any symmetric compact neighbourhood of the identity in $G$. This yields an injective continuous group homomorphism $\phi: H \rightarrow G$ with image $\langle\L\rangle$ and such that $\phi^{-1}(\L^2 \cap V^2)$ is a compact neighbourhood of the identity. By the Baire category theorem $\phi^{-1}(\L)$ has non-empty interior. The approximate subgroup $\phi^{-1}(\L)$ is therefore contained in the interior of $\phi^{-1}(\L^3)$ and commensurable to the approximate subgroup defined as the interior of $\phi^{-1}(\L^2)$. We have moreover that for any $K \subset G$ compact the subset $K \cap \L$ is covered by finitely many left-translates of $\L^2 \cap V^2$. So $\phi^{-1}(K \cap \L)$ is compact and $\phi$ is proper.

 If moreover $G$ is a Lie group, then $H$ is a Lie group as a consequence of \cite[Chapter III, \S 8, Corollary 1]{BourbakiNicolas1975LgaL}.
\end{proof}

 We will use Lemma \ref{Lemma: Closed approximate subgroups are locally stable by product} and the point of view of local groups once more later on in Section \ref{Section: Amenable}  to define - at least locally - the quotient of an ambient group by a closed approximate subgroup. This will then enable us to build local Borel sections of closed approximate subgroups (see Lemma \ref{Lemma: Borel section of closed approximate subgroups}). We will also make use of this point of view when proving the structure theorem for amenable closed approximate subgroups (Theorem \ref{Theorem: Structure amenable approximate subgroups}).

\subsection{Closed approximate subgroups of Euclidean spaces} As a first consequence we investigate the structure of closed approximate subgroups of Euclidean spaces. A key ingredient is a result due to Schreiber concerning the coarse structure of approximate subgroups of Euclidean spaces (\cite{schreiber1973approximations}). A new proof of this result was given by Fish recently \cite{fish2019extensions} (see also the generalisation to linear real soluble groups \cite{machado2019infinite}). 

\begin{theorem*}[Schreiber, \cite{schreiber1973approximations,fish2019extensions}]
For any approximate subgroup $\Xi$ in a vector space $V$ there is a vector subspace $V' \subset V$ and a compact neighbourhood of the identity $K \subset V$ such that $\Xi \subset V' + K$ and $V' \subset \Xi + K$
\end{theorem*}

We sketch here a proof making use of the structure of amenable approximate subgroups that we will establish below (Section \ref{Section: Amenable}). Let $\Vert \cdot \Vert_{\infty}$ be the sup norm on $\R^n$. Define $L:=\{x \in \mathbb{Z}^d | \exists \lambda \in \Lambda, \Vert x - \lambda \Vert_{\infty}\leq 1 \}$. We have
$ \Lambda \subset [-1;1]^n + L$ and  $L \subset [-1;1]^n + \Lambda$. Then $L$ is an approximate subgroup. By Proposition \ref{Proposition: An amenable approximate subgroup has a good model} $L+L+L+L$ has a good model. So by part (2) of Proposition \ref{Proposition: Extremal good models} there is an approximate subgroup $L' \subset L+L+L+L$ commensurable to $L$ that has a good model $f:\langle L'\rangle \rightarrow \mathbb{R}^d$. So $f$ extends to a continuous group homomorphism $f'$ from the $\mathbb{R}$-span of $L'$ to $\mathbb{R}^d$. The vector subspace we are looking for is the kernel of $f'$.

Let us now state the main result of this section: 

\begin{proposition}\label{Proposition: Classification closed approximate subgroups of Euclidean spaces}
 Let $\L$ be a closed approximate subgroup of $\R^n$. There are two vector subspaces $V_o$ and $V_d$ of $\R^n$, a uniformly discrete approximate subgroup $\L_d \subset V_d$ and a compact approximate subgroup $K \subset V_d$ such that $V_o \oplus V_d = \R^n$ and $\L$ is commensurable to $V_o + \L_d + K$. Furthermore:
 \begin{enumerate}
  \item there is a vector subspace $V_e \subset V_d$ such that we can choose $K$ to be any compact open neighbourhood of the identity in $V_e$ and $V_e \cap \L_d^2 =\{0\}$;
  \item there are a non-negative integer $m$, a linear map $\phi: \R^m \rightarrow V_d$, a subspace $V_d' \subset \R^m$ with $V_d' \cap \ker(\phi) =\{e\}$ such that $\L_d$ is contained in and commensurable to $\phi\left(\Z^m \cap \left(V_d' + [-1;1]^m\right)\right)$.  
 \end{enumerate}
\end{proposition}

\begin{proof}
While we have used additive notations in the statement, we will stick with multiplicative notations in the proof for the sake of consistency. Recall first that a consequence of Schreiber's theorem asserts that there exists a relatively compact subset $K' \subset \Xi^4$ that generates $\langle\Xi\rangle$ (e.g. \cite[Prop. 3]{machado2019infinite}). We will use this several times in what follows. Note also that the first part of (1) is a consequence of Theorem \ref{Theorem: Closed-approximate-subgroup theorem Lie group form} and the second part will follow naturally from our choice of $V_o$.

We start with two sub-cases.  Suppose first that $\L$ is uniformly discrete. Then $\langle\L\rangle$ is finitely generated, so $\langle\L\rangle \simeq \Z^m$. Take a linear map $\phi:  \R^m \rightarrow \R^n$ such that the restriction of $\phi$ to $\Z^m$ yields an isomorphism $\langle\L\rangle \simeq \Z^m$. Then (2) is a consequence of Schreiber's theorem in $\R^m$. Suppose now that $\L$ has non-empty interior. Then the interior of $\L^2$ is symmetric, contains the identity and is commensurable to $\L$. Hence, it is an open approximate subgroup commensurable to $\L$. Take $V'$ and $K$ as in Schreiber's theorem and let $W$ be any relatively compact neighbourhood of the identity. We know that $K^2W$ is covered by finitely many left-translates of $\L^2$. So $V'K \subset \L K^2W$ is covered by finitely many left-translates of $\L$. But $\L \subset V'K$ and $K$ is covered by finitely many left-translates of $W$, so $\L$ and $V'W$ are commensurable. So $V_0 := V'$ and $\Lambda_d:=\{e\}$ work.
  
 Let us go back to the general case. Let $V_o$ be the maximal vector subspace covered by finitely many left-translates of $\L$ and take $V_d$ any supplementary space. Now, $\Lambda$ is contained in $V_o\left(V_d \cap \Lambda V_o^{-1}\right)$. But $V_o \cap \L^2$ is commensurable to $V_o$ by Lemma \ref{Lemma: Intersection of approximate subgroups}. So $\L$ is covered by finitely many translates of (and, thus, commensurable to) $(V_o \cap \Lambda^2) (V_d\cap \L^2)$ according to Lemma \ref{Lemma: Intersection of approximate subgroups} again. In turn, $\L$ is commensurable to $V_o (V_d\cap \L^2)$. Let $i: L \rightarrow \R^n$ be the injective Lie group homomorphism given by Theorem \ref{Theorem: Closed-approximate-subgroup theorem general form} applied to $V_d \cap \L^2$ and let $\L'$ denote the inverse image of $V_d \cap \L^2$.  Let $L^0$ denote the connected component of the identity of $L$, then $L^0 \simeq \R^k$. We have that $(\L')^2 \cap L^0$ is an approximate subgroup with non-empty interior so it is commensurable to $V' W$ (by the above paragraph) where $V' \subset L^0$ is a vector subspace and $W$ is a compact neighbourhood of the identity in $L^0$. By construction of $V_d$ we know that $V'=\{0\}$. So $(\L')^2 \cap L^0$ is a compact neighbourhood of the identity in $L^0$. But $L$ is a torsion-free abelian Lie group and is moreover compactly generated according to the first line. Thus, $L \simeq \R^k \times \Z^l$ for some non-negative integers $k,l$. So we can identify $L$ with a closed subgroup of $ \R^k \times \R^l$ with $L^0 = \R^k \times \{0\}$. According to Schreiber's theorem there is vector subspace $V \subset \R^k \times \R^l$ and a compact neighbourhood of the identity $K \subset \R^k \times \R^l$ such that $\L' \subset V K$ and $V \subset \L' K$. But, by our discussion, $L^0 \cap V = \{0\}$. So choose a vector subspace $V'$ such that $L^0 \oplus V \oplus V'= \R^k \times \R^l$. The projection of $L$ to $V \oplus V'$ parallel to $L^0$ is then a discrete subgroup $\Gamma \subset V\oplus V'$ and we find $L = L^0 \oplus \Gamma$. Moreover, the projection of $\L'$ to $L^0$ is a bounded subset with non-empty interior. Since $\Lambda'^2 \cap L^0$ is a neighbourhood of $\{0\}$, the projection of $\L'$ to $L^0$ is covered by finitely many translates of $\Lambda'^2 \cap L^0$. So $\L'$ is commensurable to $\left(\L'^2 \cap \Gamma \right)\left( \L'^2 \cap L^0\right)$. As $i_{|\overline{\L'^2}}$ is proper, we can set $\Lambda_d=i(\overline{\Lambda'^2} \cap \Gamma)$ and $K= i( \overline{\Lambda'^2} \cap L^0)$. 
\end{proof}
The situation becomes even more striking when $\L$ is a closed approximate subgroup in a one dimensional Euclidean space:

\begin{corollary}\label{Corollary: Classification of closed approximate subgroups in R}
Let $\L$ be a closed approximate subgroup of $\R$. Then one and only one of the following is true:
\begin{enumerate}
 \item $\L$ is finite;
 \item there are real numbers $0 < a < b < \infty$ such that $[-a;a] \subset \L^2 \subset [-b;b]$;
 \item $\L$ is a uniform approximate lattice i.e. uniformly discrete and relatively dense; 
 \item there is $n \in \N$ such that $\L^n=\R$. 
\end{enumerate}
In particular, $\L$ is uniformly discrete or $\L$ has non-empty interior.
\end{corollary}

\subsection{Structure of Compact Approximate Subgroups}
We turn now to the proof of Theorem \ref{Theorem: Structure of compact approximate subgroups}. 

\begin{proof}[Proof of Theorem \ref{Theorem: Structure of compact approximate subgroups}.]

Recall that Kreitlon-Carolino proved the statement of Theorem \ref{Theorem: Structure of compact approximate subgroups} with the additional assumption that $\L$ is open (\cite[Theorem 1.25]{MR3438951}). We will show how Theorem \ref{Theorem: Structure of compact approximate subgroups} reduces to this situation. Let $L$ and $f: L \rightarrow G$ be the locally compact group and the homomorphism given by Theorem \ref{Theorem: Closed-approximate-subgroup theorem general form} and $V \subset L$ be a neighbourhood such that $f(V)=\L^2$. Then $V$ is a compact neighbourhood of the identity, so there is an open symmetric subset $\tilde{V}$ such that $V \subset \tilde{V} \subset V^2$. The subset $\tilde{V}$ is thus an open relatively compact $K^6$-approximate subgroup. But $f$ is an injective continuous homomorphism so \cite[Theorem 1.25]{MR3438951} applied to $\tilde{V}$ yields Theorem \ref{Theorem: Structure of compact approximate subgroups}.

\end{proof}

\subsection{Bohr-type Compactification} We mention yet another application of Theorem \ref{Theorem: Closed-approximate-subgroup theorem general form} that generalises further the Bohr compactification of a discrete group.

\begin{proposition}\label{Proposition: Bohr compactification}
 Let $\Gamma$ be a group that commensurate an approximate subgroup $\L \subset \Gamma$. Then there is a group homomorphism $f_0: \Gamma \rightarrow H_0$ (unique up to continuous group isomorphism $H_0 \rightarrow H_0'$) with $f_0(\L)$ relatively compact that satisfies the following universal property: 
 \begin{itemize}
  \item[$(*)$]  if $f: \Gamma \rightarrow H$ is a group homomorphism with H locally compact and $f(\L)$ relatively compact, then there is a continuous group homomorphism $\phi: H_0 \rightarrow H$ such that $f=\phi \circ f_0$. 
 \end{itemize}
\end{proposition}

\begin{proof}
 Let $\mathcal{R}$ be a set of representatives of group homomorphisms $f: \Gamma \rightarrow H$ with $H$ locally compact, dense image and $f(\L)$ relatively compact, up to the following equivalence: $f_1: \Gamma \rightarrow H_1$ and $f_2: \Gamma \rightarrow H_2$ are equivalent if there is a continuous group isomorphism $\phi: H_1\rightarrow H_2$ such that $\phi \circ f_1 = f_2$. Then the group $H_{\mathcal{R}}:=\prod_{f: \Gamma \rightarrow H_1\in \mathcal{R}} H_2$ equipped with the product topology is a Hausdorff topological group and $f_{\mathcal{R}}(\L)$ is relatively compact where $f_{\mathcal{R}}: \Gamma \rightarrow H_{\mathcal{R}}$ is the diagonal map. One readily sees that $f_{\mathcal{R}}$ satisfies the universal property ($*$). The topological group $H_{\mathcal{R}}$ need not be locally compact however. By Theorem \ref{Theorem: Closed-approximate-subgroup theorem general form} there are a locally compact group $H_0$, a group homomorphism $f_0: \Gamma \rightarrow H_0$ and a continuous group homomorphism $\phi: H_0 \rightarrow H_{\mathcal{R}}$ such that $f_{\mathcal{R}}=\phi \circ f_0$.
\end{proof}

The above proposition can be interpreted as existence of a smallest Meyer subset containing a given approximate subgroup $\L$.  Given $f_0: \Gamma \rightarrow $ as in Proposition \ref{Proposition: Bohr compactification} set $\L_0:=f_0^{-1}(\overline{f_0(\L^2)})$. Note that by construction $\overline{f_0(\L^2)}$ is a compact neighbourhood of the identity in $H_0$. So $\L_0$ is an approximate subgroup and has a good model. We see now that if $\L'$ contains $\L$ and $\L$ has a good model $f$, then $f$ must factor through $f_0$. Therefore $f(\L_0)$ is relatively compact and finitely many left-translates of $\L'$ cover $\L_0$. This discussion yields:

\begin{lemma}
 Let $\Gamma$ be a group that commensurate an approximate subgroup $\L \subset \Gamma$. Let $f_0: \Gamma \rightarrow H_0$ be as in Proposition \ref{Proposition: Bohr compactification}. Then:
 \begin{enumerate}
  \item $\L$ has a good model is and only if $f_0$ is a good model;
  \item $\L$ is Meyer subset if and only if it is commensurable to $f_0^{-1}(\overline{f_0(\L)})$.
 \end{enumerate}
\end{lemma}
\begin{remark}

Hrushovski's \cite[\S 5.8]{hrushovski2020beyond} yields moreover the following fascinating result: if $\Lambda \subset \Lambda_0$ are as above, then the numbers $n \geq 0$ such that there are pairwise non-commensurable approximate subgroups $, \Lambda \subset \Lambda_n\subset \ldots \subset \Lambda_1 \subset \Lambda_0$ are bounded.
\end{remark}
\section{Amenable approximate subgroups}\label{Section: Amenable}

\subsection{Amenable approximate subgroups: definition}

 \begin{definition}\label{Definition: Amenable approximate subgroups}
  Let $\L$ be a closed approximate subgroup of a locally compact group $G$. Define $\mathcal{B}(\Lambda)$ as the set of those Borel subsets of $G$ that are covered by finitely many left-translates of $\Lambda$. We say that $\L$ is \emph{amenable} if there exists a finitely additive measure $m$ defined on $\mathcal{B}(\Lambda)$ such that:
  \begin{enumerate}
   \item(finiteness) $0<m(\L)<\infty$ ;
   \item(left-invariance) for all $g \in G$ and $X \in \mathcal{B}(\Lambda)$, we have $m(\lambda X)=m(X)$. 
  \end{enumerate}
 \end{definition}

According to Theorem \ref{Theorem: Closed-approximate-subgroup theorem general form} compact approximate subgroups of locally compact groups are amenable and $m$ is then a Haar measure. We will see in Subsection \ref{Subsection: Approximate Subgroups in Amenable Groups} below that any closed approximate subgroup of an amenable locally compact group is amenable.

The definition above is extremely close to the definition of definably amenable approximate subgroups introduced by Massicot and Wagner in \cite{MR3345797}. There they study finitely additive measures defined on definable subsets in some structure. When $\Lambda$ is discrete, our definition is in fact a special case of definably amenable approximate subgroup (with all subsets of $\langle \Lambda \rangle$ being definable). However, it does not seem obvious how to present Definition \ref{Definition: Amenable approximate subgroups} as a special case of definably amenable approximate subgroups when $\Lambda$ is not discrete. Indeed, algebras of Borel subsets and of definable subsets behave differently with respect to set operations such as projections.

We note now that Definition \ref{Definition: Amenable approximate subgroups} is in fact \emph{local}:

\begin{lemma}\label{Lemma: Definition domain of the invariant mean}
Let $m$ be a finitely additive measure defined on the Borel subsets of $\Lambda$ and such that:
\begin{enumerate}
\item (finiteness) $m(\Lambda)=1$;
\item (local left-invariance) $m(g X) = m(X)$ whenever $g \in G, X \subset \Lambda$ and $g X \subset \Lambda$.
\end{enumerate}
Then $\Lambda$ is amenable and $m$ can be extended to a finitely additive measure as in Definition \ref{Definition: Amenable approximate subgroups}. 
\end{lemma}

\begin{proof}
Consider $X \in \mathcal{B}(\L)$ and $X_1, \ldots, X_r$ a Borel partition of $X$ such that there are $f_1,\ldots, f_r \in G$ with $f_i X_i \subset \Lambda$. We will prove that the quantity $\sum_{i=1}^r m(f_iX_i)$ depends only on $X$. Defining $\tilde{m}(X)=\sum_{i=1}^r m(f_iX_i)$ then yields the extension we are looking for. Take $Y_1,\ldots, Y_s$ a second partition with $g_1, \ldots, g_s \in G$ as above. We have 
\begin{align*}
\sum_{i=1}^r m(f_iX_i)& =  \sum_{i=1}^r m(f_iX_i) \\
                      & = \sum_{i=1}^r\sum_{j=1}^s m(f_i(X_i \cap Y_j)) \\
                      & = \sum_{j=1}^s\sum_{i=1}^r m(f_ig_j^{-1} g_j(X_i \cap Y_j)) \\
                      & = \sum_{j=1}^s\sum_{i=1}^r m(g_j(X_i \cap Y_j))\\
                      & = \sum_{i=1}^s m(g_jY_j)
\end{align*}
where we have used local left-invariance to go from the fourth to the fifth line.
\end{proof}

As an immediate corollary we find:

\begin{lemma}\label{Lemma: Approximate subgroups of amenable approximate subgroups are amenable}
Let $\L$ and $\Xi$ be closed approximate subgroups of some locally compact group. If $\L \subset \Xi$, $\L$ and $\Xi$ are commensurable and $\Xi$ is amenable, then $\L$ is amenable.  
\end{lemma}

Note that a careful study of elementary properties of invariant finitely additive measures was carried out in \cite{hrushovski2019amenability}.

\subsection{Amenable approximate subgroups of linear groups} We will exploit the strong Tits' alternative, following an idea from \cite{breuillard2011note}, to prove the following: 

\begin{lemma}\label{Lemma: Amenable approximate subgroups of linear groups are soluble}
Let $k$ be any field. Let $\Lambda$ be an amenable closed $K$-approximate subgroup of some locally compact group. Let $\psi: \overline{\Lambda^5} \rightarrow \GL_d(k)$ be a local group homomorphism i.e. $\psi(xy) = \psi(x)\psi(y)$ whenever $x,y, xy \in \overline{\Lambda^5}$. Assume that $\psi$ is continuous and has countable image. Then there is $\Lambda' \subset \Lambda^2$ commensurable with $\Lambda$ such that every finite subset of $\psi(\Lambda')$ generates a virtually soluble subgroup.
\end{lemma}

\begin{proof}
Fix $m$ an invariant finitely additive measure as in Definition \ref{Definition: Amenable approximate subgroups}. By the strong Tits alternative \cite{breuillard2008strong} there is an integer $N:=N(d)$ such that for every subset $F \subset \GL_d(k)$ either $F$ generates a virtually soluble subgroup or $F^N$ contains two elements generating a free group. We will apply this result in combination with the following:

\begin{claim} 
Let $X \subset \overline{\Lambda^5}$ be a Borel subset and $x, y \in \overline{\Lambda^5}$ be such that $\psi(x)$ and $\psi(y)$ generate a free group, and $\{x,y,x^{-1},y^{-1}\}X \subset \overline{\Lambda^5}$. Then $$m(\{x,y,x^{-1},y^{-1}\}X) \geq 3m(X) .$$
\end{claim}

Suppose first that the claim is true. Since $\overline{\Lambda^4}$ has a good model (Proposition \ref{Proposition: An amenable approximate subgroup has a good model}) there is $S \subset \overline{\Lambda^4}$ an approximate subgroup commensurable with $\Lambda$ such that $S^l \subset \overline{\Lambda^4}$ where $l \geq N\left(4\log_3(K)+1\right)$. Take a finite symmetric subset $F \subset S$ and assume for a contradiction that $\psi(F)$ does not generate a virtually soluble subgroup. According to the strong Tits alternative and the claim, 

$$ 3^{l/N}m(\Lambda) \leq m(F^l\Lambda)\leq  m(\overline{\Lambda^5}) \leq K^4m(\Lambda). $$
Hence, $3^{l/N} = 3K^4 \leq K^4$. A contradiction. So $\Lambda':= S$ works.

It remains only to prove the claim. Let $x,y \in X$ be two elements such that $F_2:=\langle \psi(x), \psi(y) \rangle$ is free. Choose $R$ a system of representatives of $F_2\setminus \GL_d(k)$. For every reduced word $w \in F_2\setminus \{e\}$ (in the letters $\psi(x), \psi(y)$) define the Borel subset $\Lambda_{w}$ as the subset of those elements $\lambda$ of $\Lambda$ such that $\psi(\lambda)=xr$ with $r \in R$ and $x \in F_2$ where $x$ starts with $w$ when written as a reduced word. In other words, $x=wv$ for some $v \in F_2$ and the last letter of the reduced word $w$ is not equal to the inverse of the first letter of the reduced word $v$. We define moreover $\Lambda_{e}:= \Lambda \cap \psi^{-1}(R)$. We have the disjoint union decomposition 
\begin{equation}
\Lambda = \bigsqcup_{w \in \{e\} \cup \{x,y,x^{-1},y^{-1}\}} \Lambda_w. \label{Eq: 1}
\end{equation}
Furthermore, we have 
\begin{equation}
\{x,y,x^{-1},y^{-1}\} \Lambda \supset \bigsqcup_{\alpha \in \{x,y,x^{-1},y^{-1}\}}\alpha\Lambda_e \sqcup \bigsqcup_{\substack{\alpha, \beta \in \{x,y,x^{-1},y^{-1}\} \\
\alpha \neq \beta^{-1}}} \alpha\Lambda_{\psi(\beta)}.  \label{Eq: 2}
\end{equation}
Therefore, a combination of (\ref{Eq: 1}) and (\ref{Eq: 2}) yields 
$$m(\{x,y,x^{-1},y^{-1}\} \Lambda) \geq 3m(\Lambda).$$

\end{proof}

In the case of characteristic $0$ fields, we obtain a stronger result thanks to the Tits' alternative (\cite[Thm. 1]{zbMATH03374014}) in characteristic $0$. 

\begin{corollary}\label{Corollary:  Amenable approximate subgroups of linear groups are soluble, characteristic 0}
With notations as in Lemma \ref{Lemma: Amenable approximate subgroups of linear groups are soluble}. If $k$ has characteristic $0$, then $\psi(\Lambda')$ generates a virtually soluble subgroup. 
\end{corollary}

We will also invoke Lemma \ref{Lemma: Amenable approximate subgroups of linear groups are soluble} in the case of a positive characteristic field. In that situation as well we will be able to draw strong information by combining it with well-known results of Tits' \cite[Lem. 2.6]{zbMATH03374014}.

\subsection{Structure of amenable approximate subgroups}
 The main result of this subsection is the following: 

\begin{proposition}\label{Proposition: Structure amenable approximate subgroups good model form}
Let $\Lambda$ be an amenable closed approximate subgroup of a $\sigma$-compact locally compact group $G$. There is $\Lambda' \subset \Lambda^4$ a closed approximate subgroup commensurable to $\Lambda$ that has a good model $f:\langle\Lambda'\rangle \rightarrow H$ such that:
\begin{enumerate}
\item $H$ is a connected Lie group;
\item $f_{\vert \overline{\Lambda'^2}}$ is continuous;
\item if $p:H \rightarrow S$ denotes the projection to the semi-simple Levi factor, then $(p \circ f)(\Lambda')$ is a neighbourhood of the identity. 
\end{enumerate}
\end{proposition} 

The first step towards Proposition \ref{Proposition: Structure amenable approximate subgroups good model form} is to prove a result in the spirit of Hrushovski's stabilizer theorem from \cite{MR2833482}:

\begin{proposition}\label{Proposition: An amenable approximate subgroup has a good model}
 Let $\L$ be a closed approximate subgroup of a $\sigma$-compact locally compact group $G$. If $\L$ is amenable, then $\overline{\L^4}$ has a continuous good model. 
\end{proposition}

Hrushovski's study of \emph{near-subgroups} (\cite{MR2833482}) immediately implies the above Proposition \ref{Proposition: An amenable approximate subgroup has a good model} when $\Lambda$ is discrete. To deal with the general case we rely on a variation of an argument due to Massicot--Wagner in \cite{MR3345797} about definably amenable approximate subgroup inspired by Sanders' \cite{MR2911137} and Croot--Sisask's \cite{MR2738997} who proved it for finite abelian approximate groups. 

\begin{lemma}\label{Lemma: Massicot-Wagner}
 Let $\L$ be an amenable closed approximate subgroup of a $\sigma$-compact locally compact group $G$ and let $k$ be a positive integer. There is an approximate subgroup $S \subset \L^2$  commensurable to $\L$ such that $S^k \subset \L^4$.
\end{lemma}

\begin{proof}
Let $\Xi \subset \L$ be Haar measurable such that $m(\Xi) \geq tm(\L)$ for some $t\in (0;1]$. Set 
 $ X(\Xi) : = \{ g \in \L^2 | m(g\Xi \cap \Xi) \geq stm(\L) \} $
 where $s = \frac{t}{2K}$. By the proof of \cite[Thm. 12]{MR3345797}, the approximate subgroup $\L$ is covered by at most $N := \lfloor \frac{1}{s} \rfloor$ left translates of $X(\Xi)$. 
 
Define now
 $$ f(t):=\inf \left\{\left. \frac{m(\Xi\L)}{m(\L)} \right| \Xi \subset \L \text{ closed},\ m(\Xi) \geq tm(\L)\right\}. $$ Note that $f$ is well defined since the product of two $\sigma$-compact subsets is a $\sigma$-compact subset, hence Borel. We also know that $f(t) \in [1;K]$ for all $t \leq 1$. 
 Take $t \geq c_{K,k}$ such that $f(\frac{t^2}{2K})\geq (1-\frac{1}{4k})f(t)$ (where we can choose $c_{K,k} = \frac{1}{(2K)^{2^n-1}}$ with $ n = \left\lceil\frac{\log(K)}{\log((1-\frac{1}{4k})^{-1})}\right\rceil$ see \cite[Lem. 11]{MR3345797}) and choose $\Xi \subset \L$ closed such that $m(\Xi) \geq tm(\L)$ and $\frac{m(\L\Xi)}{m(\L)} \geq (1 + \frac{1}{4k})f(t)$.

If $g \in X(\Xi)$ we have:
 \begin{align*}
  m(g\Xi\L \cap \Xi\L) &\geq m((g\Xi\cap \Xi)\L) \\
                   &\geq f\left(\frac{t^2}{2K}\right)m(\L) \\
                   &\geq \left(1-\frac{1}{4k}\right)f(t)m(\L) \\
                   &\geq \frac{1-\frac{1}{4k}}{1+\frac{1}{4k}}m(\Xi\L).
 \end{align*}
 Hence,
 \begin{align*}
 m(\left(g\Xi\L\right)\Delta \left(\Xi\L\right)) & \leq 2\left(1-\frac{1-\frac{1}{4k}}{1+\frac{1}{4k}}\right)m(\Xi\L) \\
 & \leq \frac{1}{k} m(\Xi\L).
 \end{align*}
 Thus, by an easy induction we have: 
 \begin{equation}
  m(\left(g_1\cdots g_k\Xi\L\right)\Delta \left(\Xi\L\right)) \leq m(\Xi\L). \label{eq:ED1} \tag{$*$}
 \end{equation}
 As a consequence, $X(\Xi)^k \subset \L^4$ and $\lfloor \frac{2K}{t} \rfloor \leq \frac{2K}{c_{K,k}}+1\approx (2K)^{2^{4\log(K)k}}$ translates of $X(\Xi)$ cover $\L$. 
\end{proof}

\begin{remark}
For the proof of Lemma \ref{Lemma: Massicot-Wagner} to work, the finitely additive measure $m$ need not be defined on all Borel subsets, but only on a certain lattice of subsets generated by $\L$, finite intersections of translates of $\L$ and certain products of such subsets. See \cite{hrushovski2019amenability}.
\end{remark}

\begin{proof}[Proof of Proposition \ref{Proposition: An amenable approximate subgroup has a good model}.]
Let $\phi: \tilde{G} \rightarrow G$ be the map given by the closed approximate subgroup theorem (Theorem \ref{Theorem: Closed-approximate-subgroup theorem general form}). We know that $\phi_{\vert\phi^{-1}(\Lambda)}$ is an homeomorphism onto its image. So $\phi^{-1}(\Lambda)$ is amenable as well. From now on we will therefore assume that $\Lambda$ has non-empty interior in $G$.
By Lemma \ref{Lemma: Massicot-Wagner} there is a closed approximate subgroup $\L_1 \subset \overline{\L^2}$ commensurable to $\L$ and such that $\overline{\L_1^{8}} \subset \overline{\L^4}$.  According to Lemma \ref{Lemma: Approximate subgroups of amenable approximate subgroups are amenable} the closed approximate subgroup $\L_1$ is amenable. We can thus build inductively a sequence of closed approximate subgroups $(\L_n)_{n\geq 0}$ commensurable to $\L$ such that $\L_0=\L$ and $(\overline{\L_{n+1}^4})^2 \subset \overline{\L_{n+1}^8} \subset \overline{\L_n^4}$ for all integers $n\geq 0$. By Theorem \ref{Theorem: Characterisation of good models} applied to the sequence $(\overline{\L_n^4})_{n\geq 0}$ we obtain that $\overline{\L^4}$ has a good model.  But now for all $n \geq 0$ the approximate subgroup $\Lambda_n$ is commensurable to $\Lambda$. According to the Baire category theorem, we have that the interior of $\overline{\Lambda_n}$ is not empty. Hence, $\overline{\L_n^4}$ is a neighbourhood of the identity. So we see that the good model built in the proof of Theorem \ref{Theorem: Characterisation of good models} - see (4) of Theorem \ref{Theorem: Characterisation of good models} - is in fact continuous.
\end{proof}

\begin{proof}[Proof of Proposition \ref{Proposition: Structure amenable approximate subgroups good model form}.]
Let us assume - as in the proof of Proposition \ref{Proposition: An amenable approximate subgroup has a good model} - that the interior of $\Lambda$ is not empty. Note that, then, $\langle \Lambda \rangle$ is an open subgroup. According to Proposition \ref{Proposition: An amenable approximate subgroup has a good model}, $\overline{\Lambda^4}$ has a continuous good model $f_0:\langle\Lambda\rangle \rightarrow H_0$. Take $W_0 \subset H_0$ a relatively compact neighbourhood of the identity such that $f_0^{-1}(W_0)\subset \overline{\Lambda^4}$. According to the Gleason--Yamabe theorem there are a symmetric compact neighbourhood of the identity $W_1 \subset W_0$ and a compact subgroup $K \subset W_1$ normal in the group $H_1$ generated by $W_1$ such that $H:=H_1/K$ is a connected Lie group without non-trivial compact normal subgroup. Write $\Lambda_1:=f_0^{-1}(W_1)$ and let $f: \langle \Lambda_1 \rangle \rightarrow H$ denote the map obtained that way. The Baire category theorem therefore implies that $\langle\Lambda_1\rangle$ is an open subgroup in $G$. Let $W_2$ be a neighbourhood of the identity in $H$ - to be chosen later - contained in the projection of $W_1$ and write $\Lambda_2:=f^{-1}(W_2)$. By Lemma \ref{Lemma: Inverse images of compact neighbourhoods are pairwise commensurable} the closed approximate subgroup $\Lambda_2$ is commensurable to $\Lambda$. So $f$ restricted to $\langle \Lambda_2 \rangle$ satisfies (1) and (2).  Let $V$ be a symmetric compact neighbourhood of the identity in $G$. Consider the approximate subgroup $A:=f(\Lambda_2 \cap V)$ (Lemma \ref{Lemma: Intersection of approximate subgroups}). Since $f$ is continuous, $A$ is compact. Let $\mathfrak{g}$ denote the Lie algebra of $G$ and $\mathfrak{a}$ denote the Lie algebra of $A$ (Theorem \ref{Theorem: Closed-approximate-subgroup theorem Lie group form}). Since $\mathfrak{a}$ is normalised by the dense subgroup $f(\langle\Lambda_2\rangle)$, the Lie subalgebra $\mathfrak{a}$ is an ideal. To prove (3) it suffices now to show that $\mathfrak{g}/\mathfrak{a}$ is soluble. There is a Lie subgroup $L$ of some $\GL_n(\mathbb{R})$ with Lie algebra $\mathfrak{g}/\mathfrak{a}$ and a local Lie group homomorphism $\phi: H \rightarrow L$. Assume $W_2$ chosen sufficiently small for $\phi$ to be defined over $\overline{W_2^5}$. The map $\psi:=\phi \circ f_{\vert \overline{\Lambda_2^5}}$ is therefore a well-defined continuous map that has countable image. According to Corollary \ref{Corollary:  Amenable approximate subgroups of linear groups are soluble, characteristic 0} there is $S$ commensurable to $\Lambda_2$ such that $\psi(S)$ generates a virtually soluble subgroup. Since $\psi(\Lambda_2)$ is dense in a subset with non-empty interior of $L$, then $\psi(S)$ is dense in a subset with non-empty interior according to the Baire category theorem. So $L$ is a virtually soluble connected Lie group, and, hence, $\mathfrak{g}/\mathfrak{a}$ is soluble. In particular, $\mathfrak{a}$ contains all the semi-simple Lie sub-algebras in $\mathfrak{g}$ which concludes.
\end{proof}

\subsection{Approximate Subgroups in Amenable Groups}\label{Subsection: Approximate Subgroups in Amenable Groups}
 Our goal in this section is to exhibit a natural family of amenable approximate subgroups i.e. to prove:
 
 \begin{proposition}\label{Theorem: Uniformly discrete approximate subgroups around an amenable approximate subgroup are amenable}
 Let $\Lambda$ be a closed approximate subgroup in a locally compact group $G$. Let $A$ be an amenable closed normal subgroup of $G$ and suppose that the projection of $\Lambda$ to $G/A$ is relatively compact. Then $\Lambda$ is amenable.
 \end{proposition} 
 
The proof consists of two steps. We first show that any neighbourhood of a normal amenable subgroup is an amenable approximate subgroup (Proposition \ref{Proposition: Amenability of neighbourhoods of amenable groups}). We then prove heredity of amenability for closed approximate subgroups (Proposition \ref{Proposition: Heredity amenability}).  

\begin{proposition}\label{Proposition: Amenability of neighbourhoods of amenable groups}
 Let $G$ be a locally compact group, let $H$ be a closed amenable normal subgroup and let $W \subset G$ be a compact symmetric neighbourhood of the identity. Then $WH$ is an amenable closed approximate subgroup of $G$. 
\end{proposition}

\begin{proof}
Let us first recall some notations and definitions (see \cite{MR0251549} for more details). Fix a left-Haar measure $\mu_G$ on $G$. Define the left- and right-translates of a function $f : G \rightarrow \R$ by $g \in G$ as the maps $_{g}f: x\mapsto f(g^{-1}x)$ and $f_g:x \mapsto f(xg)$ respectively. A function $f : G \rightarrow \R$ is \emph{right-uniformly continuous} if for any real number $\epsilon > 0$ there is a neighbourhood $U(\epsilon)\subset G$ of the identity such that for all $g \in U(\epsilon)$ and $x \in G$ we have $ |f(x)- _{g}f(x)| < \epsilon$. The set of right-uniformly continuous bounded functions on $G$ will be denoted by $C^0_{b,ru}(G)$. Likewise the set of continuous bounded functions (resp. continuous functions with compact support) on $G$ will be denoted by $C^0_{b}(G)$ (resp. $C^0_c(G)$) . One readily checks that $G$ acts continuously by left-translations on the normed vector space $C^0_{b,ru}(G)$ equipped with the norm $|| \cdot ||_{\infty}$. A linear map $F: X \rightarrow \R$ is said \emph{left-invariant} if the subspace $X \subset C^0_{b,ru}(G)$ is stable by the $G$-action and if for every $g \in G$ and $f \in X$ we have $F(_gf)=F(f)$. It is said \emph{positive} is for all $f \in X$ with $f\geq 0$ we have $F(f) \geq 0$. 

 The vector subspace of $C_{b, ru}^0(G)$ we want to consider is
$$X:=\{ f \in C^0_{b,ru}(G) | \ p(\supp(f)) \text{ is relatively compact.} \}$$ where $p : G\rightarrow G/H$ is the natural projection. We will prove the following claim.

\begin{claim}\label{Claim: Existence of a mean around an amenable group}
 There exists a non-trivial left-invariant positive linear map $m:X \rightarrow \R$.
\end{claim}

Fix $\mu_{G/H}$ a right-Haar measure on $G/H$. First of all, note that $X$ is stable under the action of $G$. Since $H$ is an amenable locally compact group there is a left-invariant mean $m_H: C^0_{b}(H)\rightarrow \R$ according to \cite[Theorem 2.2.1]{MR0251549}. This means that $m_H$ is a left-invariant positive linear functional such that for any $f \in C^0_{b}(H)$ we have $m_H(f) \leq ||f||_{\infty}$ and $m_H(\mathds{1}_{H})=1$. Take $f \in X$ and consider the map 
\begin{align*}
 \tilde{f} : G & \rightarrow \R \\
             x &\mapsto m_H\left((_xf)_{\vert H}\right).
\end{align*}
We will show that $\tilde{f}$ is continuous and invariant under left-translation by elements of $H$. Indeed, if $h,x \in H$ and $g \in G$ then $_{hg}f(x)=  (_gf)(h^{-1}x)$. But $h^{-1}x \in H$ if and only if $x \in H$. So $_{hg}f_{\vert H}=_h\left(_gf_{\vert H}\right)$ and, hence, for $x \in G$ we have
$$_h\tilde{f}(x)= \tilde{f}(h^{-1}x) = m_H\left(_{h^{-1}x}f_{\vert H}\right)=m_H\left(_{h^{-1}}\left(_xf_{\vert H}\right)\right)=m_H\left(_xf_{\vert H}\right) = \tilde{f}(x).$$
Moreover, for any $x_1,x_2 \in G$ we have 
 $$|\tilde{f}(x_1)-\tilde{f}(x_2)|=|m_H(_{x_1}f_{\vert H})-m_H(_{x_2}f_{\vert H})| \leq \vert\vert _{x_1}f - _{x_2}f\vert\vert_{\infty}.$$ But $f$ is right-uniformly continuous, so $\tilde{f}$ is continuous. Therefore, there exists a unique continuous function $f_{G/H} : G/H \rightarrow \R$ such that $\left(f_{G/H}\circ p\right)(x) =  m_H(_xf_{\vert H})$ (recall that $p$ denotes the natural projection). The map $f \mapsto f_{G/H}$ is linear, sends non-negative functions to non-negative functions and $||f_{G/H}||_{\infty} \leq ||f||_{\infty}$ for all $f \in C^0_{b,ru}(G)$ . Furthermore, we have $\supp(f_{G/H}) \subset \overline{p(\supp(f))}$, so $f_{G/H}$ is a continuous function with compact support. We are thus able to define
 \begin{align*}
  m : X & \longrightarrow \R \\
      f & \longmapsto \int_{G/H}f_{G/H}(t)d\mu_{G/H}(t).    
 \end{align*}
 The map $m$ is a positive linear map. Choose a compact neighbourhood of the identity $U \subset G$ and $f \in X$ such that $f(x)=1$ for all $x \in UH$. Then for all $x \in UH$ we have $_xf_{|H}=1$, so $f_{G/H}(p(x))=1$. This implies $m(f)\geq \mu_{G/H}(p(U))>0$ so $m$ is non-trivial. It only remains to check that $m$ is left-invariant. Take $g,x \in G$ and $f \in X$. Then
 \begin{align*}
  \left((_gf)_{G/H}\circ p\right)(x) & = m_H((_x(_gf))_{\vert H}) \\
                                     & = m_H((_{xg}f)_{\vert H}) \\
                                     & = (f_{G/H}\circ p)(xg) \\
                                     & = f_{G/H}( p(x)p(g)).
 \end{align*}
Therefore, $(_gf)_{G/H} = (f_{G/H})_{p(g)}$. But $\mu_{G/H}$ is right-invariant so 
 \begin{align*}
  m(_gf) & = \int_{G/H}(_gf)_{G/H}(t)d\mu_{G/H}(t) \\
         & = \int_{G/H}\left(f_{G/H}\right)_{p(g)}(t)d\mu_{G/H}(t) \\
         & = \int_{G/H}f_{G/H}(t)d\mu_{G/H}(t) \\
         & = m(f).
 \end{align*}
 So Claim \ref{Claim: Existence of a mean around an amenable group} is proved.
 
 Following the proof of \cite[Lemma 2.2.2]{MR0251549} one now sees that if $Y \subset L^{\infty}(G)$ is the subspace of functions $f$ supported on some Haar measurable subset $B$ with $p(B)$ relatively compact, then for any $\phi\in \mathcal{C}_c^0(G)$  taking non-negative values and such that $\int_G \phi(t) d\mu_G(t)=1$, the map 
 \begin{align*}
  m_{\phi} : Y &\longrightarrow \R \\
	      f &\longmapsto m(\phi*f),
 \end{align*}
is independent of $\phi$ and $m_{\phi}$ is non-trivial, left-invariant and positive. We now see that the map defined over Borel subsets $B \in \mathcal{B}(WH)$ by  $B \mapsto m_{\phi}(\mathds{1}_B)$ is as in Definition \ref{Definition: Amenable approximate subgroups}. So $WH$ is amenable. 
\end{proof}

We will now prove the second step. Namely:

\begin{proposition}\label{Proposition: Heredity amenability}
 Let $\L$ and $\Xi$ be two closed approximate subgroups of a first countable locally compact group $G$ and suppose that $\L \subset \Xi$. If $\Xi$ is amenable, then $\L$ is amenable. 
\end{proposition}

We start with a lemma that is essentially about building a local section of approximate subgroups. 

\begin{lemma}\label{Lemma: Borel section of closed approximate subgroups}
 Let $\L$ be a closed approximate subgroup in a first-countable locally compact group $G$. For any neighbourhood of the identity $W$ of $G$, there is a Borel subset $S \subset W$ such that $S\L^2$ has non-empty interior and $S^{-1}S \cap \L^2 =\{e\}$. Suppose moreover that $G$ is second-countable and take $X \subset G$, then there is a Borel subset $S' \subset XW$ with $(S')^{-1}S' \cap \overline{\L^2}= \{e\}$ and $X \subset S'\overline{\L^4}$.
\end{lemma}

\begin{proof}
 Let $V$ be a symmetric compact neighbourhood of the identity with $V^6$ contained in the intersection of $W$ and the open subset $U(\L^8)$ given by Lemma \ref{Lemma: Closed approximate subgroups are locally stable by product} applied to $\Lambda$ and $\Lambda^8$. Define the relation $\sim$ on $V$ by $g\sim h$ if and only if $g \in h \L^2$. We have that $\sim $ is reflexive and symmetric because $e \in \L^2$ and $\L^2$ is symmetric. In addition, given $g_1,g_2,g_3 \in V$  such that $g_1 \sim  g_2$ and $g_2 \sim  g_3$ we have $g_1 \in g_2 \L^2 \subset g_3\L^4$. So $g_3^{-1}g_1 \in \L^4 \cap V^2 \subset \L^2$, which yields $g_1\sim g_3$. Hence, $\sim $ is an equivalence relation. Write $Y$ the quotient space $V/\sim $ endowed with the quotient topology and let $q: V \rightarrow Y$ be the quotient map. Take any subset $X \subset V$ then $q^{-1}(q(X))=X\L^2 \cap V$. In particular, points of $Y$ are closed, $Y$ is compact and normal. So $Y$ is a compact metrizable space and $q$ is a continuous map. By a theorem of Federer and Morse \cite{MR7916} there is a Borel subset $S \subset V$ such that $q_{|S}$ is bijective i.e. $S^{-1}S \cap \L^2 =\{e\}$ and $V \subset S\L^2$. 

Now, since $G$ is second countable there is a sequence $(g_n)_{n\geq0}$ of elements of $X$ with $g_0=e$ such that $(g_nS\overline{\L^2})_{n\geq 0}$ covers $X$. Define inductively $S_0=S$ and $$S_{n+1}= g_{n+1}S \setminus \bigcup_{m \leq n}S_m\overline{\L^2} \subset XW$$ for all $n \geq 0$. We claim that $S_n$ is a Borel subset for all $n\geq 0$. If $n=0$ the result is clear. We proceed now by induction on $n$. By assumption $S_m$ is a Borel subset for all $m < n$. Moreover 
$$S_m^{-1}S_m \cap \overline{\L^2}^2 \subset S^{-1}S \cap \overline{\L^2}^2 \subset S^{-1}S \cap \L^2 =\{e\}$$ where the second inclusion is a consequence of $S^{-1}S \subset U(\L^8)$. Therefore the multiplication map $G \times G \rightarrow G$ is injective when restricted to the Borel subset $S_m \times \overline{\L^2}$. Which yields that $S_m\overline{\L^2}$ is Borel, and so is $S_n$. Define now $S_{\infty}:=\bigcup_{n\geq 0}S_n$. We have that $S_{\infty}^{-1}S_{\infty} \cap \overline{\L^2} = \{e\}$ and, by induction on $n$, $g_nS \subset S_{\infty}\overline{\L^2}$ for all $n\geq 0$. So $ X \subset S_{\infty}\overline{\L^4}$. 
\end{proof}

\begin{proof}[Proof of Proposition \ref{Proposition: Heredity amenability}.]
By Theorem \ref{Theorem: Closed-approximate-subgroup theorem general form}, we can assume that $\Xi$ has non-empty interior in $G$. Take a Borel section $S$ given by Lemma \ref{Lemma: Borel section of closed approximate subgroups} applied to the approximate subgroup $\L$, the subset $X=\Xi$ and $W \subset \Xi^2$. Set $S':=S^{-1}$. We have $\Xi \subset \overline{\L^4}S'$ and $\L S' \subset \overline{\Xi^3}$. In addition, the multiplication map $\L \times S' \rightarrow \L S'$ is a bijective measurable map. Hence, if $B \subset \Lambda$ is any Borel subset, then $BS'$ is Borel. For all Borel subsets $B \subset \Lambda$ define therefore $m'(B):=m(BS')$. We have that $m'$ is a locally left-invariant finitely additive measure with $m'(\Lambda) < \infty$. And we claim that $0< m'(\Lambda)$. There is indeed a finite subset $F$ of $G$ such that $\Xi \subset F \L S'$. Since $\Lambda S' \subset \Xi^3$, $F \Lambda S'$ is covered by finitely many translates of $\Xi$. By left-invariance of $m$ we find that $m'(\Lambda) \geq \frac{1}{\vert F\vert} m(\Xi) > 0$. We conclude by invoking Lemma \ref{Lemma: Definition domain of the invariant mean}.
 \end{proof}

\begin{proof}[Proof of Proposition \ref{Theorem: Uniformly discrete approximate subgroups around an amenable approximate subgroup are amenable}.]
 We have that $\L$ is contained in $WH$. But $WH$ is amenable by Proposition \ref{Proposition: Amenability of neighbourhoods of amenable groups}. So $\L$ is amenable by Proposition \ref{Proposition: Heredity amenability}.  
\end{proof}

\begin{corollary}\label{Corollary: Uniformly discrete approximate subgroups in amenable locally compact groups are amenable}
 Let $G$ be an amenable second countable locally compact group. If $\L \subset G$ is a closed approximate subgroup, then $\L$ is amenable. 
\end{corollary}

\begin{proof}
 Corollary \ref{Corollary: Uniformly discrete approximate subgroups in amenable locally compact groups are amenable} is a consequence of Proposition \ref{Theorem: Uniformly discrete approximate subgroups around an amenable approximate subgroup are amenable} applied to $G$ and $H=G$.
\end{proof}

\begin{corollary}\label{Corollary: Approximate subgroups of amenable groups are amenable}
 If $\L$ is an approximate subgroup of a countable discrete amenable group $G$, then $\L$ is amenable.
\end{corollary}

\begin{proof}
 The group $G$ equipped with the discrete topology is an amenable locally compact group. Moreover, $\L$ is obviously a uniformly discrete approximate subgroup in this topology. So Corollary \ref{Corollary: Approximate subgroups of amenable groups are amenable} is a consequence of Corollary \ref{Corollary: Uniformly discrete approximate subgroups in amenable locally compact groups are amenable}.
\end{proof}

\subsection{Structure of approximate subgroups in amenable groups} 

The Meyer-type theorem (Theorem \ref{Theorem: Approximate lattices in amenable locally compact groups are contained in model sets}) is the most immediate consequence of the above results: 

\begin{proof}[Proof of Theorem \ref{Theorem: Approximate lattices in amenable locally compact groups are contained in model sets}.]
 Since $\L$ is uniformly discrete it is amenable by Corollary \ref{Corollary: Uniformly discrete approximate subgroups in amenable locally compact groups are amenable} and so the approximate subgroup $\overline{\L^4}=\L^4$ has a good model, $f$ say, by Proposition \ref{Proposition: An amenable approximate subgroup has a good model}. Since $\Lambda$ is an approximate lattice, $\L^4$ is a strong approximate lattice by \cite[2.11]{bjorklund2019borel}. Thus, the approximate subgroup $\L^4$ is contained in and commensurable to a model set by Proposition \ref{Proposition: An approximate lattice has a good model if and only if it is a model set}. (In fact, the proof of Proposition \ref{Proposition: An approximate lattice has a good model if and only if it is a model set} reveals that $\L^4\ker(f) \subset \L^8$ is a model set).
\end{proof}

We also derive easily Theorem \ref{Theorem: Structure amenable approximate subgroups}:
\begin{proof}[Proof of Theorem \ref{Theorem: Structure amenable approximate subgroups}.] Let $\Lambda'$ and $f: \langle \Lambda' \rangle \rightarrow H$ be the approximate subgroup and good model provided by Proposition \ref{Proposition: Structure amenable approximate subgroups good model form} applied to $\Lambda$. Let $R$ denote the soluble radical of $H$ and let $S:=H/R$ be the Levi factor. As in the proof of Proposition \ref{Proposition: Structure amenable approximate subgroups good model form} let $A$ denote an approximate subgroup defined as $\Lambda' \cap V$ where $V$ is a symmetric neighbourhood of the identity in $G$ and recall that $p \circ f(A)$ is a neighbourhood of the identity in $S$ where $p: H \rightarrow S$ denotes the natural map. Write finally $W$ a neighbourhood of the identity in $H$ such that $\Lambda'=f^{-1}(W)$. We will show that taking $N:= \ker f$ and $\Lambda_{sol}:=f^{-1}(W_1 \cap R)$ works where $W_1$ is a neighbourhood of the identity to be chosen later. First of all, $N$ is indeed a closed subgroup and $\Lambda_{sol}$ is a closed approximate subgroup (Corollary \ref{Corollary: Definition of good models is independent of the choice of neighbourhood}). Moreover, the proof of Theorem \ref{Theorem: Closed-approximate-subgroup theorem Lie group form} tells us that $\langle \Lambda_{sol} \rangle /N$ equipped with the topology obtained via Theorem \ref{Theorem: Closed-approximate-subgroup theorem general form} is a soluble Lie group. So (1) and (2) are proved - the bound on the dimension in (2) excepted. To prove (3), notice that $p \circ f(\Lambda')$ is relatively compact and $p \circ f(A)$ is a neighbourhood of the identity. Therefore, there is $n \geq 0$ such that $p \circ f(\Lambda') \subset p \circ f(A^n)$. In other words,  for all $\lambda \in \Lambda'$ there is $a \in A^n$ such that $f(a^{-1}\lambda) \subset R$. Moreover, $f(a^{-1}\lambda)$ is contained in $W^{n+1}$. If we choose $W_1:=W^{n+1}$ , then $f(a^{-1}\lambda) \in W_1 \cap R$ i.e. $\Lambda' \subset A^n \Lambda_{sol}$. Finally, if we assume, as we may, that $f^{-1}(W^{n+1}) \subset \overline{\Lambda^4}$, then (3) is proved.

It remains to prove the bound in (2). Recall that $\langle \Lambda' \rangle \cap \Lambda^2$ is a $K^3$-approximate subgroup commensurable to $\Lambda'$ (Lemma \ref{Lemma: Intersection of approximate subgroups}). So $\overline{f(\langle \Lambda' \rangle \cap \Lambda^2)}=:W_2$ is a $K^3$-approximate subgroup with non-empty interior in $H$. Let $N$ denote the nilpotent radical of $G$ i.e. the largest nilpotent connected normal subgroup of $H$. We have that $W_2^2 \cap R$ is a neighbourhood of the identity and a $K^9$-approximate subgroup in $R$ (Lemma \ref{Lemma: Intersection of approximate subgroups}). Since the maximal normal compact subgroup of $N$ is characteristic in $N$, it is normal in $H$. Hence, $N$ has no normal compact subgroup and, hence, compact subgroup altogether. According to \cite[Thm. 1.2]{jing2021nonabelian}, the non-compact dimension - the dimension of $R$ minus the dimension of its maximal compact subgroup - of $R$ is at most $9\log_2(K)$. In particular, $N$ has dimension at most $9\log_2(K)$. Let $\mathfrak{n}$ denote the Lie algebra of $N$. Then $R$ acts on $\mathfrak{n}$ via the restriction of the adjoint action $\Ad$. The connected component of the identity of the kernel of this action is contained in $N$ and its image is contained in $\GL(\mathfrak{n})$. Since a compact connected soluble subgroup of $\GL(\mathfrak{n})$ is a torus, it has dimension at most $9\log_2(K)$ as well. Hence, compact subgroups of $R$ have dimension at most $9\log_2(K)$. Overall, $R$ has dimension at most $18\log_2(K)$. 
\end{proof}

As an immediate corollary of the proof of (2):

\begin{corollary}
With notations as in Proposition \ref{Proposition: Structure amenable approximate subgroups good model form} and suppose that $\Lambda$ is a $K$-approximate subgroup. The radical of $H$ has dimension at most $18 \log K$.
\end{corollary}

Under additional assumptions, the structure of amenable approximate subgroups appears even more strikingly: 

\begin{corollary}\label{Corollary: structure of discrete amenable approximate subgroups}
Let $\Lambda$ be an amenable discrete approximate subgroup of a $\sigma$-compact locally compact group $G$. Then there is $\Lambda' \subset \Lambda^4$ and a closed subgroup $N \subset \Lambda'$ normalised by $\Lambda'$ such that $\langle\Lambda'\rangle/N$ is soluble.   
\end{corollary}

\begin{corollary}\label{Corollary: structure of amenable approximate subgroups of totally disconnected groups}
Let $\Lambda$ be an amenable closed approximate subgroup of a totally disconnected $\sigma$-compact locally compact group $G$. Then there is $\Lambda' \subset \Lambda^4$ and a closed subgroup $N \subset \Lambda'$ normalised by $\Lambda'$ such that $\langle\Lambda'\rangle/N$ is soluble.   
\end{corollary}
 
Another key consequence is the fact that approximate lattices are uniform, when considered in the right ambient group: 

\begin{proposition}\label{Proposition: Approximate lattices in amenable S-adic Lie groups are uniform} 
 Let $\L$ be an approximate lattice in an amenable locally compact second countable group $G$.
 \begin{enumerate}
 \item there is $L \subset G$ a closed subgroup such that $\Lambda$ is covered by finitely many cosets of $L$ and $\Lambda^2 \cap L$ is uniform in $L$; 
 \item if $G$ is an $S$-adic Lie group (i.e. locally a finite product of real and p-adic Lie groups, see e.g. \cite{MR3186334}), then $\L$ is a uniform approximate lattice.
 \end{enumerate} 
\end{proposition}

\begin{proof}
We know that $\Lambda$ is amenable (Proposition \ref{Proposition: Heredity amenability}). Let $\Lambda'$ and $f: \langle \Lambda' \rangle \rightarrow H$ be given by Proposition \ref{Proposition: Structure amenable approximate subgroups good model form}. Since $\Lambda'$ is discrete, $H$ is soluble. Hence, $G \times H$ is amenable. By Proposition \ref{Proposition: An approximate lattice has a good model if and only if it is a model set}, the graph $\Gamma_f$ of $f$ is a lattice in $G \times H$. According to \cite[Prop. 3.4]{MR3186334}, $\Gamma_f$ is a uniform lattice in $\overline{\Gamma_f(G^0 \times H)}$ where $G^0$ denotes the connected component of the identity in $G$. But $\overline{\Gamma_f(G^0 \times H)} = L \times H$ for some closed subgroup $L$ of $G$ and $L$ satisfies the conclusions of (1).

Let us prove (2). There are $\L'$ commensurable to $\L$ and $f:\langle \L'\rangle \rightarrow H$ a good model of $\L'$ with dense image and target a connected Lie group without non-trivial normal compact subgroups by Proposition \ref{Theorem: Uniformly discrete approximate subgroups around an amenable approximate subgroup are amenable}. So the graph $\Gamma_f \subset G \times H$ of $f$ is a lattice by Proposition \ref{Proposition: An approximate lattice has a good model if and only if it is a model set} and there is a symmetric compact neighbourhood of the identity $W_0 \subset H$ such that $\L' : = P_0(G,H, \Gamma_f, W_0)$. Since $G$ is normal, closed and amenable in $G \times H$, we know by \cite[Th. 6.6]{MR3186334} that there is closed normal subgroup $H' \subset H$ such that $\Gamma':=\Gamma_f \cap G \times H'$ is a uniform lattice in $G \times H'$. So $\L'$ contains the uniform approximate lattice $P_0(G,H',\Gamma', W_0 \cap H')$. Hence, $\L$ is uniform. 
\end{proof}

We are also able to prove finite generation of discrete approximate subgroups of soluble Lie groups - extending a classical result concerning discrete subgroups (see \cite[Proposition 3.8]{raghunathan1972discrete}):

\begin{proposition}\label{Proposition: Uniformly discrete approximate subgroups in soluble groups are finitely generated}
 Let $R$ be a connected soluble Lie group. If $\L \subset R$ is a uniformly discrete approximate subgroup, then $\langle\L\rangle$ is finitely generated. 
\end{proposition}

\begin{proof}
 According to Corollary \ref{Corollary: Approximate subgroups of amenable groups are amenable} and Proposition \ref{Proposition: An amenable approximate subgroup has a good model} the approximate subgroup $\L^4$ has a good model. According to Lemma \ref{Lemma: Miscellaneous good models} there is an approximate subgroup $\L' \subset \L^4$ that has a good model $f : \langle\L'\rangle \rightarrow H$ with dense image and target a connected Lie group. Since $\langle\L'\rangle$ is soluble we obtain that $H$ is soluble. Now, the graph of $f$, denoted by $\Gamma_f$, is a discrete subgroup of the connected soluble Lie group $R \times H$ (Lemma \ref{Lemma: The graph of a discrete approximate subgroup is discrete}) and therefore $\Gamma_f$ is finitely generated by  \cite[Proposition 3.8]{raghunathan1972discrete}. Let $F_1 \subset \L'$ be a finite set of generators of $\langle\L'\rangle$ and $F_2 \subset \langle\L\rangle$ be a finite subset such that $\L \subset F_2\L'$. Then $F_1 \cup F_2$ is a finite set that generates $\langle\L\rangle$. 
\end{proof}

\section{Generalisation of theorems of Mostow and Auslander}\label{Section: Generalisation of theorems of Mostow and Auslander}

\subsection{Intersections of approximate lattices and closed subgroups}
We will show a general theorem about intersections of approximate lattices and closed subgroups. Proposition \ref{Proposition: Intersection approximate lattice with a closed subgroup} is close in spirit to a classical fact about lattices (see for instance \cite[Theorem 1.13]{raghunathan1972discrete}). See also \cite{bjorklund2018spectral} for other results around this topic in the framework of strong (and) uniform approximate lattices.

We start with a related property concerning measures on the hull: 

\begin{proposition}\label{Proposition: Intersection strong approximate lattice with a closed subgroup}
Let $G$ be a second countable locally compact group, $X_0 \subset G$ be a uniformly discrete and $H \subset G$ be a closed normal subgroup. Write $p: G \rightarrow G/H$ be the natural projection and suppose that $\Omega_{X_0}$ admits a proper $G$-invariant Borel probability measure $\nu_0$. If $p(X_0)$ is uniformly discrete, then there is $X \in \Omega_{X_0}$ such that $\Omega_{X \cap H}$ admits a proper $H$-invariant Borel probability measure.  
\end{proposition}

\begin{proof}
Let $\mathcal{P}(X_0,H)$ denote the set of $H$-invariant Borel probability measures. Since $\nu_0 \in \mathcal{P}(X_0,H)$ and is proper, we have by the usual Krein--Millman argument that there exists $\nu_1 \in \mathcal{P}(X_0,H)$ which is ergodic and proper. Then $\nu_1$ is a Borel probability measure on the compact metrizable space $X_0$, so $\nu_1$ has a well-defined support $K$. The compact subset $K$ is $H$-invariant with $\nu_1(K)=1$ and for any open subset $U \subset \Omega_{X_0}$ we have $\nu_1(U \cap K) = 0$ if and only if $U \cap K = \emptyset$. Thus, according to e.g. \cite[Proposition 2.1.7]{zimmer2013ergodic}, there is $X_1 \in K$ such that $K = \overline{H\cdot X_1}$. Furthermore, $X_1 \neq \emptyset$ since $\overline{H\cdot \emptyset} = \{\emptyset\}$ and $\nu_1(\{\emptyset\})=0$. Choose now $x_1 \in X_1$. Then $e \in x_1^{-1}X_1 \subset X_0^{-1}X_0$ by \cite[Lemma 4.6]{bjorklund2016approximate}. Moreover, the subgroup $H$ is normal so $\nu_2:=\left(x_1^{-1}\right)_{*}\nu_1$ is an $H$-invariant ergodic Borel probability measure with $\nu_2(\{\emptyset\})=0$ and support $x_1^{-1}K= \overline{H\cdot X_2}$ where $X_2=x_1^{-1}X_1$. Define the map 
 \begin{align*}
  \pi : \overline{H\cdot X_2} & \longrightarrow \mathcal{C}(H) \\
            X       & \longmapsto X \cap H.
 \end{align*}
We see that $\pi$ is $H$-equivariant. We claim moreover that $\pi$ is continuous. Since $p(X_0)$ is uniformly discrete, there is an open subset $U \subset G$ such that $\overline{X_0^{-1}X_0}H \cap U = H$. Note in addition that $p(X) \subset p(\overline{X_0^{-1}X_0})$ for all $X \in \overline{H\cdot X_2}$.  So for any open subset $V \in H$ we have
\begin{align*}
 \pi^{-1}(U^V) & = \pi^{-1}(\{ Y \in  \mathcal{C}(H) | Y \cap V \neq \emptyset \}) \\
               & = \{ X \in \overline{H\cdot X_2}  | X\cap \left(U \cap W\right) \neq \emptyset \} \\
               & = U^{U \cap W}\cap \overline{H\cdot X_2},
\end{align*}
where $W \subset G$ is any open subset such that $H \cap W = V$. Likewise for any compact subset $L \subset H$ we have 

\begin{align*}
 \pi^{-1}(U_L) & = \pi^{-1}(\{ Y \in  \mathcal{C}(H) | Y \cap L = \emptyset \}) \\
               & = \{ X \in \overline{H\cdot X_2}  | X \cap L = \emptyset \} \\
               & = U_{L}\cap \overline{H\cdot X_2} ,
\end{align*}
where we consider $L \subset H \subset G$. So $\pi$ is indeed a continuous map. Thus, $\pi(\overline{H\cdot X_2})$ is a compact subset of $\mathcal{C}(H)$ and $\pi(X_2)=X_2 \cap H$ has dense orbit in $\pi(\overline{H\cdot X_2})$. So $\pi(\overline{H\cdot X_2}) = \Omega_{X_2 \cap H}$. Set $\nu_3 := \pi_{*}\left((\nu_2)_{\vert \overline{H\cdot X_2}}\right)$ where $(\nu_2)_{\vert \overline{H\cdot X_2}}$ is the restriction of the measure $\nu_2$ to its support $\overline{H\cdot X_2}$ which is a well defined $H$-invariant ergodic Borel probability measure since $\Omega_{X_0}$ is metric compact. Then $\nu_3$ is a $H$-invariant ergodic Borel probability measure on $\Omega_{X_2 \cap H}$. Suppose now that $\nu_3(\{\emptyset\})> 0$, then $\nu_3(\{\emptyset\})=1$ by ergodicity. Thus, $\pi^{-1}\left( \{\emptyset\}\right)$ is an $H$-invariant compact co-null subset of $\overline{H\cdot X_2}$. Which means $\pi^{-1}\left( \{\emptyset\}\right) = \overline{H\cdot X_2}$ because $\overline{H\cdot X_2}$ is the support of $\nu_2$. Therefore $\pi(X_2)=\emptyset$. A contradiction. Hence, $\nu_3(\{\emptyset\})=0$ so $\nu_3$ is a proper $H$-invariant Borel probability measure on $\Omega_{X_2\cap H}$. So $X:=X_2$ works.
\end{proof}

\begin{proposition}\label{Proposition: Intersection approximate lattice with a closed subgroup}
 Let $\L$ be a uniformly discrete approximate subgroup of a locally compact group $G$. Assume that $H$ is a closed subgroup of $G$ such that $p(\L)$ is locally finite where $p: G \rightarrow G/H$ is the natural map. We have:
 \begin{enumerate}
  \item if $\L$ is a uniform approximate lattice, then $\L^2 \cap H$ is a uniform approximate lattice in $H$;
  \item if $\L$ is a strong approximate lattice, $G$ is second countable and $H$ is normal and amenable, then $\L^2\cap H$ is a strong approximate lattice in $H$;
  \item if $\Lambda$ is an approximate lattice, $G$ is second countable and $H$ is normal and amenable, then $\L^2\cap H$ is a strong approximate lattice in $H$. 
 \end{enumerate}
\end{proposition}

\begin{proof}
 We will first prove (1). We know that $\L^2 \cap H$ is an approximate subgroup according to Lemma \ref{Lemma: Intersection of approximate subgroups}. Moreover, since $\L$ is uniformly discrete so is $\L^2\cap H$. We must prove that $\L^2 \cap H$ is relatively dense in $H$. Let $K \subset G$ be a compact subset such that $K\L = G$. Since $p(\L)$ is locally finite there is $F \subset \L$ finite such that $K^{-1}H \cap \L  \subset F H$. Take any $h \in H$ then there are $\lambda \in \L$ and $k \in K$ such that $k\lambda = h$. Which implies $\lambda \in K^{-1}H \cap \L$ and we can find $f \in F$ such that $f^{-1}\lambda \in H \cap \L^2$. Therefore, $h \in KF \left(H \cap \L^2\right)$.
 
 Let us move on to the proof of (2). Again, note that $\L^2\cap H$ is uniformly discrete and is an approximate subgroup according to Lemma \ref{Lemma: Intersection of approximate subgroups}. By Proposition \ref{Proposition: Intersection strong approximate lattice with a closed subgroup} there is $P \subset \Lambda^2$ such that $\Omega_{P \cap H}$ admits a proper $H$-invariant Borel probability measure. If $H$ is amenable we know according to (the proof of) \cite[Corollary 2.11]{bjorklund2019borel} that the approximate subgroup $\L^2 \cap H \supset P_2 \cap H$ is a strong approximate lattice in $H$.

Let us show (3). We will rely on an equivalent definition of approximate lattices established in \cite[App. A]{hrushovski2020beyond}. Since $p(\Lambda)$ is locally finite and is an approximate subgroup, it is uniformly discrete. Therefore, we may choose $F_{G/H}$ a measurable subset of positive Haar measure (possibly infinite) such that the multiplication map $p(\Lambda) \times F_{G/H} \rightarrow G/H$ is one-to-one. Take a measurable section of $F_{G/H}$ in $G$ i.e. a Borel subset $\tilde{F}_{G/H} \subset G$ such that the projection from $\tilde{F}_{G/H} \subset G$ to $F_{G/H}$ is bijective. Let now $F_H \subset H$ be any Borel subset of positive measure such that the multiplication $\Lambda^2 \cap H \times F_H \rightarrow H$ is one-to-one. We first notice that the multiplication map $\Lambda \times F_H\tilde{F}_{G/H} \rightarrow G$ is one-to-one. Indeed, take $\lambda_1, \lambda_2 \in \Lambda, \tilde{f}_1, \tilde{f}_2 \in \tilde{F}_{G/H}$ and  $f_1, f_2 \in F_N$ such that $\lambda_1f_1\tilde{f}_1 =\lambda_2f_2\tilde{f}_2$. Projecting to $G/H$ we see that $\tilde{f}_1=\tilde{f}_2$ and $p(\lambda_1)=p(\lambda_2)$. So $\lambda_2^{-1}\lambda_1f_1=f_2$. But $\lambda_2^{-1}\lambda_1 \in \Lambda^2 \cap H$ so $f_1=f_2$ and $\lambda_1=\lambda_2$. By \cite[Prop. A.2]{hrushovski2020beyond} we see that $F_H\tilde{F}_{G/H}$ has finite Haar measure. Thus, $F_H$ and $F_{G/H}$ have finite Haar measure in $H$ and $G/H$ respectively. So by \cite[Prop. A.2]{hrushovski2020beyond} again $\Lambda^2 \cap H$ and $p(\Lambda)$ are approximate lattices. 
\end{proof}

\begin{remark}
 We use Proposition \ref{Proposition: Intersection approximate lattice with a closed subgroup} in the companion paper \cite{machado2020apphigherrank} to define and study a notion of irreducibility for approximate lattices. In the language of \cite{machado2020apphigherrank}, part (2) reduces to showing that $\L^2 \cap H$ is a $\star$-approximate lattice, a notion close to the one of strong approximate lattice.
\end{remark}

We prove next a converse of sorts of Proposition \ref{Proposition: Intersection approximate lattice with a closed subgroup}:

\begin{lemma}\label{Lemma: Partial converse to ``Intersection approximate lattice with a closed subgroup''}
 Let $\L$ be a uniformly discrete approximate subgroup of a locally compact group $G$. Assume that $H$ is a closed subgroup of $G$ such that $\L^2\cap H$ is a uniform approximate lattice in $H$. Then $p(\L)$ is a locally finite subset of $G/H$ where $p:G\rightarrow G/H$ is the natural map. 
\end{lemma}

\begin{proof}
 Let $K \subset G/H$ be a compact subset. Then there is a compact subset $L \subset G$ such that $p(L)=K$. Since $\L^2\cap H$ is relatively dense in $H$ there is a compact subset $L' \subset G$ such that $LH \subset L'\left(\L^2 \cap H\right)$. Take $\lambda \in \L \cap LH$. Since $\lambda \in \left(L'\left(\L^2 \cap H\right)\right)$, we have $ \lambda \in  \left(\left(L' \cap \L^3\right)\left(\L^2 \cap H\right)\right)$. So $p(\L)\cap K \subset p(L' \cap \L^3)$ which is indeed finite. 
\end{proof}

As a first application we investigate the intersections of uniform approximate lattices with centralisers: 

\begin{corollary}\label{Corollary: Intersection of a uniform approximate lattice with centralisers}
 Let $\L$ be a uniform approximate lattice in a locally compact group $G$. Then for all $\gamma \in \langle\L\rangle$ the approximate subgroup $\L^2\cap C(\gamma)$ is a uniform approximate lattice in $C(\gamma)$ the centraliser of $\gamma$. Moreover, if $G$ is a Lie group and $\langle\L\rangle$ is dense in $G$, then $\L^2 \cap Z(G)$ is a uniform approximate lattice in $Z(G)$ the centre of $G$. 
\end{corollary}

\begin{proof}
 Let $n \geq 0$ be an integer such that $\gamma \in \L^n$ and consider the map
 \begin{align*}
  \varphi : G & \longrightarrow G \\ 
            g & \longmapsto g\gamma g^{-1}.
 \end{align*}
Then $\varphi$ factors as $\varphi = \psi \circ p$ where $\psi : G/C(\gamma) \rightarrow G$ is a continuous injective map and $p: G \rightarrow G/C(\gamma)$ is the natural map. But $\varphi(\L) \subset \L^{n+2}$ so is locally finite. Hence, $p(\L)$ is locally finite as well. By part (1) of Proposition \ref{Proposition: Intersection approximate lattice with a closed subgroup} we deduce that $\L^2 \cap C(\gamma)$ is a uniform approximate lattice in $C(\gamma)$. 

Now if $G$ is a Lie group and $\langle\L\rangle$ is dense in $G$, then $Z(G) = \bigcap_{\gamma \in \langle\L\rangle} C(\gamma).$ But $Z(G)$ is a Lie group and so are the $C(\gamma)$'s. Thus, there are $\gamma_1,\ldots,\gamma_r \in \langle\L\rangle$ such that $\dim(Z(G))= \dim(\bigcap_i C(\gamma_i))$. Consider now the map 

 \begin{align*}
  \varphi : G & \longrightarrow G^r \\ 
            g & \longmapsto (g\gamma_1 g^{-1},\ldots,g\gamma_rg^{-1}).
 \end{align*}
 As above $\varphi$ factors as $\varphi=\psi \circ p$ with $\psi: G/\left(\bigcap_i C(\gamma_i)\right) \rightarrow G^r$ an injective and continuous map and $p: G \rightarrow G/\left(\bigcap_i C(\gamma_i)\right)$ the natural map. But $\varphi(\L) \subset \prod_{1\leq i \leq r} \L^{n+2}$ where $n$ is a positive integer such that $\{\gamma_1,\ldots,\gamma_r\} \subset \L^n$. Thus, $\varphi(\L)$ is locally finite and so is $p(\L)$. By part (1) of Proposition \ref{Proposition: Intersection approximate lattice with a closed subgroup} we deduce that $\L^2 \cap \bigcap_i C(\gamma_i)$ is a uniform approximate lattice in $\bigcap_i C(\gamma_i)$. But $Z(G)$ is an open subgroup of $\bigcap_i C(\gamma_i)$ so $p'(\L^2 \cap \bigcap_i C(\gamma_i))$ is obviously locally finite where $p': \bigcap_i C(\gamma_i) \rightarrow \left(\bigcap_i C(\gamma_i)\right)/Z(G)$ is the natural map. By part (1) of Proposition \ref{Proposition: Intersection approximate lattice with a closed subgroup} once again we have that 
 $$\left(\L^2 \cap \bigcap_i C(\gamma_i)\right)^2 \cap Z(G) \subset \L^4 \cap Z(G)$$
 is a uniform approximate lattice in $Z(G)$. By Lemma \ref{Lemma: Intersection of approximate subgroups} we find that $\L^2 \cap Z(G)$ is a uniform approximate lattice in $Z(G)$.
\end{proof}

\subsection{Borel density for approximate lattices}\label{Section: Borel density for approximate lattices}

The Borel density theorem asserts that lattices in simple algebraic groups are Zariski-dense. The usual route to show Borel density-type theorems for groups with finite co-volume is to start by proving that the subgroup considered has \emph{property (S)} (see e.g. \cite{10.2307/1970150}).

\begin{definition}[Definition 1.1, \cite{10.2307/1970150}]\label{Definition: Property (S)}
 Let $G$ be a locally compact group. A closed subset $X \subset G$ has \emph{property (S)} if for all neighbourhoods $W \subset G$ of the identity and all $g \in G$ there is $n \in \N$ such that $g^n \in W X W$.
\end{definition}

Approximate lattices have property (S). Hence, they exhibit similar density properties. 

\begin{proposition}\label{Proposition: Approximate lattices have property (S)}
 Let $G$ be a locally compact second countable group. We have:
 \begin{enumerate}
  \item if $\L$ is an approximate lattice in $G$, then $\L^2$ has property $(S)$ (\cite[A.11]{hrushovski2020beyond}) ;
  \item if $X$ is a closed subset and $\Omega_{X}$ has a proper $G$-invariant Borel probability measure, then $\overline{X^{-1}X}$ has property $(S)$.
 \end{enumerate}
\end{proposition}

While the (2) will not be needed, it serves as an illustration of the method to prove both (1) and (2).
\begin{proof}  
   Let us prove (2). By assumption there is a proper $G$-invariant Borel probability measure $\nu$ on $\Omega_{X}$. If $W$ is any symmetric neighbourhood of the identity, then the open subset $U^{W}$ satisfies $\nu\left(U^{W}\right) > 0$. Indeed, $U^{W}$ is open and 
   $$\Omega_{X}\setminus\{\emptyset\} =\bigcup_{g\in G} U^{gW} =\bigcup_{g\in G} gU^{W}.$$
   Since $G$ is second countable we can find $D \subset G$ countable such that
   $$\Omega_{X}\setminus\{\emptyset\} =\bigcup_{d\in D} dU^{W}.$$
   But $\nu(\Omega_{X}\setminus\{\emptyset\}) = 1$ so there is $d \in D$ such that $0 < \nu(dU^{W}) = \nu(U^{W})$. Therefore, for any $g \in G$ there is an integer $1 \leq n < \left(\nu\left(U_{W}\right)\right)^{-1}$ such that $\nu(U_{W} \cap g^nU_{W} ) > 0$. So we can find $P \in U_{W} \cap g^nU_{W}$. Thus, $P \cap W \neq \emptyset$ and $P \cap g^nW \neq \emptyset$. That implies that $P^{-1}P \cap W g^n W \neq \emptyset$. But $P^{-1}P \subset \overline{X^{-1}X}$ so $g^n \in W \overline{X^{-1}X} W$.
\end{proof}

% In the next section we will use the following consequence of the Borel density theorem in simple groups:
% 
%\begin{proposition}[\cite{raghunathan1972discrete} and \S II.4\cite{MR1090825}]\label{Proposition: A consequence of Borel density for approximate lattices}
% Let $G$ be a simple Lie group or the $k$-points of an almost simple algebraic group over a local field $k$. Suppose that $G$ is not compact and let $\Gamma$ be a group with property (S). If $H$ is a closed (in the Lie or Zariski topology) soluble subgroup of $H$ normalised by $\Gamma$, then $H$ is discrete. 
%\end{proposition}

\begin{remark}
Borel density for approximate lattices was first investigated in \cite{bjorklund2019borel}. Their method was completely different however. The proof of Proposition \ref{Proposition: Approximate lattices have property (S)} has the advantage of yielding a short proof that can be directly applied to cases not covered by  \cite{bjorklund2019borel}, for instance approximate lattices in S-adic Lie groups.
\end{remark}

\subsection{Proof of Theorem \ref{Theorem: Auslander's theorem for approximate lattices}} 

We will make use of the Tits alternative over local fields (see \cite{zbMATH03374014}). For lack of an exact reference we include a proof relying on \cite{zbMATH03374014}: 

\begin{lemma}\label{Lemma: Tits alternative local fields}
Let $k$ be a local field. Let $\Gamma$ be a Zariski-dense subgroup in $\mathbb{G}(k)$ the $k$-points of an almost simple algebraic group $\mathbb{G}$ defined over $k$. Suppose that all finitely generated subgroups of $\Gamma$ are virtually soluble, then $\Gamma$ is not Zariski-dense. 
\end{lemma}

\begin{proof}
When $k$ has characteristic $0$, the Tits' alternative \cite[Thm 1]{zbMATH03374014} implies that $\Gamma$ is virtually soluble. So it cannot be Zariski-dense. Suppose that $k$ has positive characteristic. Suppose that $\Gamma$ is Zariski-dense. If $\Gamma$ contains a finite subset $X$ that generates an infinite subgroup, then choose one such that the Zariski-closure $H_X$ of $\langle X \rangle$ has maximal dimension. By maximality, one finds that the Zariski-connected component of the identity of $H_X$ is normalised by $\Gamma$. Therefore, $H_X = \mathbb{G}(k)$. But $\langle X \rangle$ is virtually soluble. A contradiction. So every finite subset generates a finite subgroup.  by \cite[Prop. 2.8]{zbMATH03374014} in combination with the fact that there are only finitely many elements in $k$ that are algebraic over the prime field of $k$, $\Gamma$ is finite. So $\Gamma$ is not Zariski-dense.
\end{proof}

\begin{proof}[Proof of Theorem \ref{Theorem: Auslander's theorem for approximate lattices}.]
 Let $W$ be a relatively compact neighbourhood of the identity in $G$. Define $\Lambda_A:= \Lambda^2 \cap W^{-1}WA$. The subset $\Lambda_A$ is an approximate subgroup that is commensurated by $\langle \Lambda \rangle$. Indeed, first notice that both $\Lambda^2$ and $W^{-1}WA$ are commensurated by $\langle \Lambda \rangle$, then apply Lemma \ref{Lemma: Intersection of commensurable subsets}. We know in addition that $\Lambda_A$ is an amenable discrete approximate subgroup (Proposition \ref{Proposition: Heredity amenability}). Let $p_S:G \rightarrow S$ be the natural projection to any simple factor $S$ of $G/A$. Now Lemma \ref{Lemma: Amenable approximate subgroups of linear groups are soluble} provides an approximate subgroup $\Lambda_A'$ commensurable with $\Lambda_A$ such that every finite subset of $p_S(\Lambda_A')$ generates a virtually soluble subgroup.  By Lemma \ref{Lemma: Tits alternative local fields}, the group generated by $p_S(\Lambda_A')$ is not Zariski-dense. In particular, the connected component $H$ of the identity of the Zariski-closure of $p_S(\Lambda_A')$ (\cite[Theorem 17]{bjorklund2019borel}) is a proper subgroup of $S$. But it is normalised by $\Comm_S(p_S(\Lambda_A'))$ which contains $p_S(\langle \Lambda \rangle)$. Since $p_S(\langle \Lambda \rangle)$ is supposed Zariski-dense, $H$ is trivial and $p_S(\Lambda_A)$ is finite. As a conclusion, the projection of $\Lambda_A$ to $G/A$ must be finite i.e. the projection of $\Lambda$ to $G/A$ is uniformly discrete. 
\end{proof}

\begin{proof}[Proof of Corollary \ref{Corollary: Auslander's theorem for approximate lattices}.]
 We have that $\langle \Lambda \rangle$ has property (S) (Proposition \ref{Proposition: Approximate lattices have property (S)}). By the Borel density theorem, we have that the projection of $\langle \Lambda \rangle$ to any simple factor is Zariski-dense. So we can invoke Theorem \ref{Theorem: Auslander's theorem for approximate lattices} and conclude that the projection of $\Lambda$ to $G/A$ is uniformly discrete. By Proposition \ref{Proposition: Intersection approximate lattice with a closed subgroup} we obtain the desired result.
\end{proof}

\section{Acknowledgement}

I am indebted to my PhD supervisor, Emmanuel Breuillard, for his encouragements and advice. I am deeply grateful to Tobias Hartnick and Ehud Hrushovski for many enlightening discussions. I would also like to thank Anand Pillay and Krzysztof Krupi\'{n}ski for pointing out the model-theoretic origin of Theorem \ref{Theorem: Characterisation of good models short version}. 

%\bibliographystyle{plain}
%\bibliography{biblio.bib}

\begin{thebibliography}{10}

\bibitem{MR152607}
Louis Auslander.
\newblock On radicals of discrete subgroups of {L}ie groups.
\newblock {\em Amer. J. Math.}, 85:145--150, 1963.

\bibitem{MR3136260}
Michael Baake and Uwe Grimm.
\newblock {\em Aperiodic order. {V}ol. 1}, volume 149 of {\em Encyclopedia of
  Mathematics and its Applications}.
\newblock Cambridge University Press, Cambridge, 2013.
\newblock A mathematical invitation, With a foreword by Roger Penrose.

\bibitem{MR3186334}
Yves Benoist and Jean-Fran\c{c}ois Quint.
\newblock Lattices in {$S$}-adic {L}ie groups.
\newblock {\em J. Lie Theory}, 24(1):179--197, 2014.

\bibitem{bjorklund2016approximate}
Michael Bj\"{o}rklund and Tobias Hartnick.
\newblock {Approximate lattices}.
\newblock {\em Duke Math. J.}, 167(15):2903--2964, 2018.

\bibitem{bjorklund2018spectral}
Michael Bj{\"o}rklund and Tobias Hartnick.
\newblock Spectral theory of approximate lattices in nilpotent {L}ie groups.
\newblock {\em Mathematische Annalen}, 384(3):1675--1745, 2022.

\bibitem{bjorklund2016aperiodic}
Michael Bj\"{o}rklund, Tobias Hartnick, and Felix Pogorzelski.
\newblock Aperiodic order and spherical diffraction, {I}: auto-correlation of
  regular model sets.
\newblock {\em Proc. Lond. Math. Soc. (3)}, 116(4):957--996, 2018.

\bibitem{bjorklund2017aperiodic}
Michael Bj{\"o}rklund, Tobias Hartnick, and Felix Pogorzelski.
\newblock Aperiodic order and spherical diffraction, ii: translation bounded
  measures on homogeneous spaces.
\newblock {\em Mathematische Zeitschrift}, 300(2):1157--1201, 2022.

\bibitem{bjorklund2019borel}
Michael Bj\"{o}rklund, Tobias Hartnick, and Thierry Stulemeijer.
\newblock Borel density for approximate lattices.
\newblock {\em Forum Math. Sigma}, 7:Paper No. e40, 27, 2019.

\bibitem{10.2307/1970150}
Armand Borel.
\newblock Density properties for certain subgroups of semi-simple groups
  without compact components.
\newblock {\em Annals of Mathematics}, 72(1):179--188, 1960.

\bibitem{MR0358652}
N.~Bourbaki.
\newblock {\em \'{E}l\'{e}ments de math\'{e}matique. {T}opologie
  g\'{e}n\'{e}rale. {C}hapitres 1 \`a 4}.
\newblock Hermann, Paris, 1971.

\bibitem{BourbakiNicolas1975LgaL}
Nicolas Bourbaki.
\newblock {\em Lie groups and {L}ie algebras / {N}icolas {B}ourbaki. {P}art I,
  chapters 1-3.}
\newblock Actualites scientifiques et industrielles. H ; Addison-Wesley, Paris
  : Reading, Mass, 1975.

\bibitem{breuillard2008strong}
Emmanuel Breuillard.
\newblock A strong tits alternative.
\newblock {\em arXiv preprint arXiv:0804.1395}, 2008.

\bibitem{MR3267520}
Emmanuel Breuillard.
\newblock {Geometry of locally compact groups of polynomial growth and shape of
  large balls}.
\newblock {\em Groups Geom. Dyn.}, 8(3):669--732, 2014.

\bibitem{breuillard2011note}
Emmanuel Breuillard, Ben Green, and Terence Tao.
\newblock A note on approximate subgroups of gl\_n (c) and uniformly
  nonamenable groups.
\newblock {\em arXiv preprint arXiv:1101.2552}, 2011.

\bibitem{MR3090256}
Emmanuel Breuillard, Ben Green, and Terence Tao.
\newblock {The structure of approximate groups}.
\newblock {\em Publ. Math. Inst. Hautes \'{E}tudes Sci.}, 116:115--221, 2012.

\bibitem{brooks1981some}
Robert Brooks.
\newblock {Some remarks on bounded cohomology}.
\newblock In {\em Riemann surfaces and related topics: Proceedings of the 1978
  Stony Brook Conference}, volume~97, pages 53--63, 1981.

\bibitem{MR2574741}
Pierre-Emmanuel Caprace and Nicolas Monod.
\newblock Isometry groups of non-positively curved spaces: discrete subgroups.
\newblock {\em J. Topol.}, 2(4):701--746, 2009.

\bibitem{MR3438951}
Pietro~Kreitlon Carolino.
\newblock {\em The {S}tructure of {L}ocally {C}ompact {A}pproximate {G}roups}.
\newblock ProQuest LLC, Ann Arbor, MI, 2015.
\newblock Thesis (Ph.D.)--University of California, Los Angeles.

\bibitem{MR2738997}
Ernie Croot and Olof Sisask.
\newblock A probabilistic technique for finding almost-periods of convolutions.
\newblock {\em Geom. Funct. Anal.}, 20(6):1367--1396, 2010.

\bibitem{MR7916}
H.~Federer and A.~P. Morse.
\newblock Some properties of measurable functions.
\newblock {\em Bull. Amer. Math. Soc.}, 49:270--277, 1943.

\bibitem{MR139135}
J.~M.~G. Fell.
\newblock A {H}ausdorff topology for the closed subsets of a locally compact
  non-{H}ausdorff space.
\newblock {\em Proc. Amer. Math. Soc.}, 13:472--476, 1962.

\bibitem{fish2019extensions}
Alexander Fish.
\newblock Extensions of {S}chreiber's theorem on discrete approximate subgroups
  in {$\Bbb R^d$}.
\newblock {\em J. \'{E}c. polytech. Math.}, 6:149--162, 2019.

\bibitem{MR2680491}
Isaac Goldbring.
\newblock Hilbert's fifth problem for local groups.
\newblock {\em Ann. of Math. (2)}, 172(2):1269--1314, 2010.

\bibitem{MR0251549}
Frederick~P. Greenleaf.
\newblock {\em Invariant means on topological groups and their applications}.
\newblock Van Nostrand Mathematical Studies, No. 16. Van Nostrand Reinhold Co.,
  New York-Toronto, Ont.-London, 1969.

\bibitem{MR2833482}
Ehud Hrushovski.
\newblock {Stable group theory and approximate subgroups}.
\newblock {\em J. Amer. Math. Soc.}, 25(1):189--243, 2012.

\bibitem{hrushovski2020beyond}
Ehud Hrushovski.
\newblock Beyond the {L}ascar group.
\newblock {\em arXiv preprint arXiv:2011.12009}, 2020.

\bibitem{hrushovski2019amenability}
Ehud Hrushovski, Krzysztof Krupi{\'n}ski, and Anand Pillay.
\newblock Amenability, connected components, and definable actions, 2021.

\bibitem{jing2021nonabelian}
Yifan Jing, Chieu-Minh Tran, and Ruixiang Zhang.
\newblock A nonabelian brunn-minkowski inequality, 2021.

\bibitem{machado2020approximate}
Simon Machado.
\newblock {Approximate lattices and Meyer sets in nilpotent Lie groups}.
\newblock {\em Discrete Analysis}, 1:18 pp., 2020.

\bibitem{machado2020apphigherrank}
Simon Machado.
\newblock Approximate lattices in higher-rank semi-simple groups.
\newblock {\em arXiv preprint arXiv:2011.01835}, 2020.

\bibitem{machado2022discrete}
Simon Machado.
\newblock {\em Discrete approximate subgroups of {L}ie groups}.
\newblock PhD thesis, University of Cambridge, 2022.

\bibitem{machado2019infinite}
Simon Machado.
\newblock Infinite approximate subgroups of soluble {L}ie groups.
\newblock {\em Math. Ann.}, 382(1-2):285--301, 2022.

\bibitem{MR3345797}
Jean-Cyrille Massicot and Frank~O. Wagner.
\newblock Approximate subgroups.
\newblock {\em J. \'{E}c. polytech. Math.}, 2:55--64, 2015.

\bibitem{meyer1972algebraic}
Yves Meyer.
\newblock {\em {Algebraic numbers and harmonic analysis}}, volume~2.
\newblock Elsevier, 1972.

\bibitem{moody1997meyer}
Robert~V Moody.
\newblock {Meyer sets and their duals}.
\newblock {\em NATO ASI Series C Mathematical and Physical Sciences-Advanced
  Study Institute}, 489:403--442, 1997.

\bibitem{MR0289713}
G.~D. Mostow.
\newblock {Arithmetic subgroups of groups with radical}.
\newblock {\em Ann. of Math. (2)}, 93:409--438, 1971.

\bibitem{MR1406004}
Peter~J. Olver.
\newblock Non-associative local {L}ie groups.
\newblock {\em J. Lie Theory}, 6(1):23--51, 1996.

\bibitem{raghunathan1972discrete}
M.~S. Raghunathan.
\newblock Discrete subgroups of {L}ie groups.
\newblock pages ix+227, 1972.

\bibitem{MR2911137}
Tom Sanders.
\newblock {Approximate groups and doubling metrics}.
\newblock {\em Math. Proc. Cambridge Philos. Soc.}, 152(3):385--404, 2012.

\bibitem{schreiber1973approximations}
Jean-Pierre Schreiber.
\newblock {Approximations diophantiennes et problemes additifs dans les groupes
  ab{\'e}liens localement compacts}.
\newblock {\em Bull. Soc. Math. France}, 101:297--332, 1973.

\bibitem{shalom2013commensurated}
Yehuda Shalom and George~A Willis.
\newblock Commensurated subgroups of arithmetic groups, totally disconnected
  groups and adelic rigidity.
\newblock {\em Geometric and Functional Analysis}, 23(5):1631--1683, 2013.

\bibitem{MR2501249}
Terence Tao.
\newblock {Product set estimates for non-commutative groups}.
\newblock {\em Combinatorica}, 28(5):547--594, 2008.

\bibitem{zbMATH03374014}
Jacques Tits.
\newblock Free subgroups in linear groups.
\newblock {\em J. Algebra}, 20:250--270, 1972.

\bibitem{tointon_2019}
Matthew C.~H. Tointon.
\newblock {\em Introduction to {A}pproximate {G}roups}.
\newblock London Mathematical Society Student Texts. Cambridge University
  Press, 2019.

\bibitem{MR2015025}
Kroum Tzanev.
\newblock Hecke {$C^*$}-algebras and amenability.
\newblock {\em J. Operator Theory}, 50(1):169--178, 2003.

\bibitem{MR0165031}
V.~I. U\v{s}akov.
\newblock {Topological {$\overline{FC}$}-groups}.
\newblock {\em Sibirsk. Mat. \v{Z}.}, 4:1162--1174, 1963.

\bibitem{weil1938espaces}
Andr{\'e} Weil.
\newblock {\em {Sur les espaces {\`a} structure uniforme et sur la topologie
  g{\'e}n{\'e}rale}}.
\newblock Paris, 1938.

\bibitem{zimmer2013ergodic}
Robert~J Zimmer.
\newblock {\em {Ergodic theory and semisimple groups}}, volume~81.
\newblock Springer Science \& Business Media, 2013.

\end{thebibliography}

\end{document}